\documentclass[11pt]{article}
\usepackage{amsmath}
\usepackage{amsfonts}
\usepackage{color}
\usepackage{latexsym}
\usepackage{tablefootnote}
\usepackage{epsfig}

\usepackage{amsmath}
\usepackage{amsfonts}
\usepackage{latexsym}
\usepackage{amssymb}

\newtheorem{thm}{Theorem}[section]
\newtheorem{theorem}[thm]{Theorem}

\newtheorem{lemma}[thm]{Lemma}

\newtheorem{proposition}[thm]{Proposition}

\newtheorem{definition}[thm]{Definition}

\newtheorem{remark}[thm]{Remark}

\newcommand{\beq}{\begin{equation}}
\newcommand{\eeq}{\end{equation}}
\newcommand{\beqa}{\begin{eqnarray}}
\newcommand{\eeqa}{\end{eqnarray}}
\newcommand{\beqas}{\begin{eqnarray*}}
\newcommand{\eeqas}{\end{eqnarray*}}
\newcommand{\bi}{\begin{itemize}}
\newcommand{\ei}{\end{itemize}}

\newcommand{\vgap}{\vspace{.1in}}

\def\proof{\noindent{\bf Proof}: \ignorespaces}

\def\endproof{{\ \hfill\hbox{%
      \vrule width1.0ex height1.0ex
    }\parfillskip 0pt}\par}

\newcommand{\tos}{\rightrightarrows}
\newcommand{\R}{\mathbb{R}}

\newcommand{\cL}{{\cal L}}
\newcommand{\cX}{{\cal X}}

\newcommand{\cZ}{{\cal Z}}
\newcommand{\Z}{{\cal Z}}
\newcommand{\lam}{{\lambda}}

\newcommand{\cK}{{\cal K}}
\newcommand{\norm}[1]{\left\Vert#1\right\Vert}

\newcommand{\inner}[2]{\langle #1,#2\rangle}

\newcommand{\ttheta}{{\tilde \theta}}

\newcommand{\dom}{\mathrm{dom}\,}
\newcommand{\Dom}{\mathrm{Dom}\,}

\newcommand{\tx}{\tilde x}

\newcommand{\tz}{\tilde z}

\begin{document}

\pagestyle{plain}
\setcounter{page}{1}

\title{Complexity of the relaxed Peaceman-Rachford splitting method for
the sum of two maximal strongly monotone operators }

\author{Renato D.C. Monteiro\footnote{School of Industrial and Systems Engineering, Georgia Institute of Technology, Atlanta, GA 30332-0205. (email:
monteiro@isye.gatech.edu). The work of this author was partially supported by NSF Grant CMMI-1300221.},\ \ Chee-Khian Sim\footnote{Department of Mathematics, University of Portsmouth, Lion Gate Building, Lion Terrace, Portsmouth PO1 3HF. (email: chee-khian.sim@port.ac.uk).  This research was made possible through support by Centre of Operational Research and Logistics, University of Portsmouth.}}

\date{November 3, 2016 (Revised:  November 5, 2017)}

\maketitle

\begin{abstract}
This paper considers the relaxed Peaceman-Rachford (PR) splitting method for finding an approximate solution of a monotone
inclusion whose underlying operator consists of the sum of two maximal strongly monotone operators.
Using general results obtained in the setting of  a non-Euclidean hybrid proximal extragradient framework, {\color{black}{we extend a previous convergence result on  the
iterates generated by the relaxed PR splitting method, as well as establish new pointwise and ergodic convergence rate results for the method whenever an associated
relaxation parameter is within a certain interval.}}  An example is also discussed to demonstrate that the iterates may not converge
when the relaxation parameter is outside this interval.
\end{abstract}

\section{Introduction}\label{sec:intro}

In this paper, we consider the relaxed Peaceman-Rachford (PR) splitting method for solving the monotone inclusion
\beq \label{eq:mainincl}
0 \in (A+B)(u)
\eeq
where $A: \cX \tos \cX$ and $B: \cX \tos \cX$ are maximal $\beta$-strongly monotone (point-to-set) operators for some
$\beta \ge 0$
(with the convention that $0$-strongly monotone means simply monotone, and $\beta$-strongly monotone with $\beta > 0$
means strongly monotone in the usual sense).
Recall that the relaxed PR splitting method is given by
\beq \label{eq:seqqq}
x_{k} = x_{k-1} + \theta (J_{B}(2 J_{A}(x_{k-1}) - x_{k-1}) - J_{A}(x_{k-1})),
\eeq
where $\theta>0$ is a fixed relaxation parameter and $J_T :=(I+T)^{-1}$. The special case of the relaxed PR splitting method in which
 $\theta = 2$ is known as
the Peaceman-Rachford (PR) splitting method and the one with $\theta=1$ is the
widely-studied Douglas-Rachford (DR) splitting method.
Convergence results for them are studied for example
in \cite{Bauschke,Bauschke1,Bauschke2,Bauschke3,Combettes2,Eckstein,Facchinei,MR551319}.



The analysis of the relaxed PR splitting method for the case in which $\beta=0$ has been undertaken in a number of papers
which are discussed in this paragraph.
Convergence of the sequence of iterates generated by the  relaxed PR splitting method is well-known when $\theta<2$
(see for example \cite{Bauschke,Combettes1,Facchinei}) and, according to \cite{Giselsson}, its limiting behavior for the case in which $\theta \ge 2$ is not known.
{\color{black}{We actually show in Subsection \ref{subsection2} that
the sequence \eqref{eq:seqqq} does not necessarily converge when $\theta \geq 2$.}}
 An $\mathcal{O}(1/\sqrt{k})$ (strong) pointwise convergence rate result is established in \cite{He}  for the relaxed PR splitting
method when $\theta \in (0,2)$.
Moreover, when $A=\partial f$ and $B=\partial g$ where $f$ and $g$ are  proper lower semi-continuous convex functions,
papers \cite{Davis3,Davis,Davis2} derive 
strong pointwise (resp., ergodic)
convergence rate bounds for
the relaxed PR method  when $\theta \in (0,2)$ (resp., $\theta \in (0,2]$) under different assumptions on the functions.
Assuming only $\beta$-strong monotonicity of $A = \partial f$, where $\beta > 0$, some smoothness property on $f$, and maximal monotonicity of $B$,
\cite{Giselsson} shows that  the relaxed PR
splitting method has linear convergence rate for $\theta \in (0,2+\tau)$ for some $\tau>0$.
Linear rate of convergence of the relaxed PR splitting method and its two special cases, namely, the DR splitting and PR splitting methods,
are established in \cite{Bauschke1,Bauschke2,Bauschke3,Davis2,Giselsson2,Giselsson,MR551319} under relatively strong
assumptions on $A$ and/or $B$ (see also Table \ref{table1}).

 {\color{black}{This paper  assumes that $\beta \geq 0$, and hence its analysis applies to the case in which
both $A$ and $B$ are monotone ($\beta=0$) and the case in which both $A$ and $B$ are strongly monotone ($\beta > 0$).}}
This paragraph discusses papers dealing with the latter case.
Paper \cite{Dong}
establishes  convergence of the sequence generated by the relaxed PR splitting method for any
$\theta \in (0, 2 + \beta)$ and, under some strong assumptions on $A$ and $B$,
establishes its linear convergence rate.  We complement the convergence results in \cite{Dong} by showing that for
$\theta = 2 + \beta$, the sequence of iterates generated by the relaxed PR splitting method also converge,
and describe an instance showing its nonconvergence when $\theta \geq \min\{2 + 2\beta, 2 + \beta + 1/\beta \}$.
Moreover, we establish strong pointwise and ergodic convergence rate results
(Theorems \ref{thm:pointwise2} and \ref{thm:main}) 
for the relaxed PR splitting method when
$\theta \in (0, 2 + \beta)$ and $\theta \in (0, 2 + \beta]$, respectively.

Finally, by imposing strong assumptions requiring one of the operators to be strong monotone and
one of them to be Lipschitz (and hence point-to-point), \cite{Davis2,Giselsson2,Giselsson} establish linear convergence rate of the
relaxed PR splitting method.
As opposed to these papers, the assumptions in \cite{Dong} and this paper do not imply the
operators $A$ or $B$  to be point-to-point.

Our analysis of the relaxed PR splitting method for solving (\ref{eq:mainincl}) is based on viewing it as an inexact proximal point method, more
specifically, as an instance of a non-Euclidean hybrid proximal extragradient (HPE) framework for solving the monotone inclusion problem.  
The proximal point method, proposed by Rockafellar \cite{Rockafellar1}, is a classical iterative scheme for solving the latter problem.
Paper \cite{Solodov} introduces an Euclidean version of the HPE framework which is an inexact version of the proximal point method based
on a certain relative error criterion.  Iteration-complexities of the latter framework are established in \cite{Monteiro} (see also \cite{monteiro2011complexity}).
Generalizations of the HPE framework to the non-Euclidean setting are studied in \cite{Goncalves,Kolossoski,Solodov1}.
Applications of the HPE framework can be found for example in \cite{He3,He1,monteiro2011complexity,Monteiro}.


This paper is organized as follows. Section \ref{sec:bas} describes basic concepts and notation used in the paper.  Section \ref{sec:HPE} discusses
 the non-Euclidean HPE framework which is used to the study the convergence properties of the relaxed PR splitting method
in Sections  \ref{sec:PRmethod} and \ref{discussions}.
Section \ref{sec:PRmethod} derives convergence rate bounds for the relaxed Peaceman-Rachford (PR) splitting method.
{\color{black}{Section \ref{discussions}, which consists of two subsections,
discusses a convergence result of the relaxed PR splitting method in the first
subsection and provides an example showing that its iterates may not
converge when $\theta \geq \min\{2 + 2\beta, 2 + \beta + 1/\beta \}$ in the second subsection. 
Finally, Section \ref{numerical} discusses the
numerical performance of the relaxed PR splitting method for solving
 the weighted Lasso minimization problem.
Section \ref{conclusion} gives some concluding remarks.}}

\section{Basic concepts and notation}
\label{sec:bas}

This section presents some definitions, notation and terminology which will be used in the paper.


{\color{black}{We denote the set of real numbers by $\mathbb{R}$ and the set of non-negative real numbers by $\R_+$.}}
Let $f$ and $g$ be functions with the same domain and  whose values are in {\color{black}{$\R_+$}}.
We write that  $f(\cdot) = {\Omega} (g(\cdot))$ if there exists constant $K > 0$  such that $f(\cdot) \geq K g(\cdot)$.
Also, we write $f(\cdot) = {\Theta}(g(\cdot))$ if $f(\cdot) = {\Omega}(g(\cdot))$ and $g(\cdot) = {\Omega}(f(\cdot))$.

Let  $\cZ$ be a finite-dimensional
real vector space with
inner product denoted by $\inner{\cdot}{\cdot}$ (an example of $\cZ$ is $\R^n$ endowed with the standard inner product)
and let $\|\cdot\|$ denote an arbitrary
seminorm in $\cZ$.  Its dual (extended) seminorm, denoted by $\|\cdot\|_*$, is defined as
 $\|\cdot\|_*:=\sup\{ \inner{\cdot}{z}:\|z\|\leq 1\}$. {\color{black}{It is easy to see that
\beq \label{eq:Cau-Sc}
\inner {z}{v} \le \|z\| \|v\|_* \quad \forall z,v \in \cZ.
\eeq
}}
 
 The following straightforward  result states some basic properties of the dual seminorm associated with a matrix seminorm.
Its proof can be found for example in {\color{black}{Lemma A.1(b) of \cite{Sicre}.}}


\begin{proposition}\label{propdualnorm}
Let $A:\cZ\to \cZ$ be a self-adjoint positive semidefinite linear operator and consider
 the seminorm $\|\cdot\|$ in $\cZ$ given by $\|z\|= \langle A z, z\rangle ^{1/2}$ for every $z\in \cZ$. Then,
 $\dom  \|\cdot\|_*=\mbox{Im}\;(A)$ and $\| Az\|_*=\|z\|$ for every $z\in \cZ$.
\end{proposition}

Given a set-valued operator $T:\cZ\tos \cZ$,
its domain is denoted by
 $\mbox{Dom}(T):=\{z\in \cZ : T(z)\neq \emptyset\}$
 and  its {inverse} operator $T^{-1}:\cZ\tos \cZ$ is given by
$T^{-1}(v):=\{z :  v\in T(z)\}$.  The graph of $T$ is defined by ${\rm{Gr}}(T) := \{(z,t) : t \in T(z) \}$.
The operator $T$ is said to be   monotone if
\[
\inner{z-z'}{t-t'}\geq 0\quad \forall (z,t), \, (z',t') \in {\rm{Gr}}(T).
\]
Moreover, $T$ is maximal monotone if it is monotone and,
additionally, if $T'$ is a monotone operator such that $T(z)\subset T'(z)$ for every $z\in \cZ$, then $T=T'$.
The sum $ T+T':\cZ\tos \cZ$ of two set-valued
operators $T,T':\cZ\tos \cZ$ is defined by
$(T+T')(z):= \{t+t' \in \cZ:t\in T(z),\; t'\in T'(z)\}$ for every $z\in \cZ$.
Given a scalar $\varepsilon\geq0$, the
 {$\varepsilon$-enlargement} $T^{[\varepsilon]}:\cZ\tos \cZ$
 of a monotone operator $T:\cZ\tos \cZ$ is defined as
\begin{align}
\label{eq:def.eps}
 T^{[\varepsilon]}(z)
 :=\{t\in \cZ : \inner{t-t'}{z-z'}\geq -\varepsilon,\;\forall z' \in \cZ, \, \forall t'\in T(z')\} \quad \forall z \in \cZ.
\end{align}

%
%


\section{A non-Euclidean hybrid proximal extragradient framework}\label{sec:HPE}

{\color{black}{This section discusses the non-Euclidean hybrid proximal extragradient (NE-HPE) framework and describes
its associated convergence and iteration complexity results.
The results of the section will be used in Sections \ref{sec:PRmethod} and \ref{discussions} to study the convergence
and iteration complexity properties of the relaxed PR splitting method \eqref{eq:seqqq}.
It contains two subsections. The first one describes a class of distance generating functions introduced in \cite{Goncalves}
and derives some of its basic properties. The second one describes the NE-HPE framework and its
corresponding convergence and iteration complexity results.}}

\subsection{A class of distance generating functions}

We start by introducing a class of distance generating functions (and its corresponding Bregman distances)
which is needed for the presentation of the NE-HPE framework in Subsection \ref{subsec:HPE}.

\begin{definition}\label{def:defw0}
For a given convex set {\color{black}{$Z\subset \cZ$, a seminorm $\| \cdot \|$ in $\cZ$}} and scalars $0 < m \le M$, we let $\mathcal{D}_Z(m,M)$ denote the class of real-valued
functions  $w$  which are differentiable on $Z$
and satisfy
\begin{align}\label{eq:strongly-dist}
w(z')-w(z) - \inner{\nabla w(z)}{z'-z} \geq \frac{{m}}{2}\|z-z'\|^2 \quad \forall  z,z' \in  Z,\\
\label{eq:a1}
\|\nabla w(z)-\nabla w(z')\|_* \leq {M}\|z-z'\| \quad \forall z, z' \in Z.
\end{align}
A function $w \in \mathcal{D}_Z(m,M)$ is referred to as a distance generating function with respect to the seminorm $\|\cdot\|$ and its associated
Bregman distance {\color{black}{$dw: Z \times Z \rightarrow \R$}} is defined as
\begin{equation}\label{def_d}
(dw)(z';z) = (dw)_{z}(z') := w(z')-w(z)-\langle \nabla w(z),z'-z\rangle  \quad \forall z, z' \in Z.
\end{equation}
\end{definition}

Throughout our presentation, we use the second notation $(dw)_{z}(z')$ instead of the first one $(dw)(z';z)$
although the latter one makes it clear that $(dw)$ is a function of two arguments, namely, $z$ and $z'$.
Clearly, it follows from \eqref{eq:strongly-dist} that $w$ is a convex function on $Z$ which is
in fact $m$-strongly convex on $Z$ whenever $\|\cdot\|$ is a norm.


The following simple result summarizes
the main identities about the Bregman distance $(dw)$.

\begin{lemma} \label{lm:basic-ident}
For some convex set $Z\subset \cZ$ and scalars $0 < m \le M$, let $w \in \mathcal{D}_Z(m,M)$ be given.
Then, the following identities hold for every $z,z' \in Z$:
\begin{align}
\nabla (dw)_{z}(z') &= - \nabla (dw)_{z'}(z) = \nabla w(z') - \nabla w(z), \label{grad-d} \\
\label{equacao_taylor}
(dw)_{v}(z') - (dw)_{v}(z) &= \langle \nabla (dw)_{v}(z), z'-z\rangle + (dw)_{z}(z'), \quad \forall v \in Z\\
\label{lipsc}
\frac{m}2 \|z-z'\|^2 &\le (d{w})_{z}(z')  \leq \frac{M}{2}\|z-z'\|^2, \\
\label{eq:789}
 \|\nabla(dw)_{z'}(z)\|_*^2 &\leq  \frac{2 M^2}{m} \min \{ (dw)_{z}(z') ,  (dw)_{z'}(z) \};
\end{align}
\end{lemma}

\proof
Identities \eqref{grad-d} and \eqref{equacao_taylor} follow straightforwardly from the definition of the Bregman distance in \eqref{def_d}.
{\color{black}{The first inequality in \eqref{lipsc} follows easily from (\ref{eq:strongly-dist}) and the definition of $(dw)_{z}(z')$ in (\ref{def_d}). The second inequality in (\ref{lipsc}) follows from \eqref{eq:Cau-Sc}, (\ref{eq:a1}),
the definition of $(dw)_{z}(z')$ in (\ref{def_d}),
and the identity
\begin{eqnarray*}
w(z^\prime) - w(z) = \int_0^1 \langle \nabla w(z + t(z^\prime - z)), z^\prime - z \rangle dt \quad z, z^\prime \in Z.
\end{eqnarray*}
 It is easy to see that \eqref{eq:789} immediately follows from
\eqref{eq:a1}, \eqref{grad-d} and \eqref{lipsc}.
}}
\endproof

Note that if the seminorm in Definition \ref{def:defw0} is a norm, then \eqref{eq:strongly-dist}  implies that
$w$ is strongly convex on $Z$, in which case the corresponding $dw$ is said to be nondegenerate on $Z$.
However, since Definition \ref{def:defw0} does not necessarily assume that $\|\cdot\|$ is a norm, it admits the possibility of
$w$ being not strongly convex on $Z$, or equivalently, $dw$ being degenerate on $Z$.


The following result gives some useful properties of distance generating functions.

\begin{lemma}\label{basicassu}
For some convex set $Z\subset \cZ$ and scalars $0 < m \le M$, let $w \in \mathcal{D}_Z(m,M)$ be given.
Then,
for every $l \ge 1$ and  $z_0,z_1,\ldots,z_l \in Z$, we have
\begin{equation}\label{eq:56}
(dw)_{z_0}(z_l) \le \frac{lM}{m} \sum_{i=1}^l \min \{ (dw)_{z_{i-1}}(z_i) ,  (dw)_{z_i}(z_{i-1}) \}.
\end{equation}
\end{lemma}
\proof
By \eqref{lipsc}, the triangle inequality for norms and the fact that
the $1$-norm of an $l$-vector is bounded by $\sqrt{l}$ times its $2$-norm, we have
\begin{align*}
(dw)_{z_0}(z_l) &\le \frac{M}2 \|z_l-z_0\|^2 \le \frac{M}2 \left ( \sum_{i=1}^l \|z_i - z_{i-1}\| \right)^2
\le \frac{l M}2 \sum_{i=1}^l \|z_i - z_{i-1}\|^2 
\end{align*}
which clearly implies \eqref{eq:56} due to the first inequality in \eqref{lipsc}.
\endproof

\subsection{The NE-HPE framework}
\label{subsec:HPE}

 This subsection describes the NE-HPE framework and its
corresponding convergence and iteration complexity results.

Throughout this subsection, we assume that scalars $0<m \le M$, convex set $Z \subset \Z$, seminorm $\|\cdot\|$  and
distance generating function $w \in   \mathcal{D}_Z(m,M)$ with respect to $\|\cdot\|$ are given.
Our problem of interest in this section is the MIP
\begin{align}\label{eq:inc.p}
 0\in T(z)
\end{align}
where
$T:\cZ\tos \cZ$ is a maximal monotone operator satisfying the following conditions:
\begin{itemize}
\item[\bf A0)] $\Dom(T)\subset Z$;
\item[\bf A1)]  the solution set $T^{-1}(0)$ of~\eqref{eq:inc.p} is nonempty.
\end{itemize}

%

We now state a non-Euclidean HPE (NE-HPE) framework  for solving the MIP \eqref{eq:inc.p} which generalizes
its Euclidean counterparts studied in the literature (see for example in  \cite{Monteiro,Monteiro1,Solodov}).

\vgap
\vgap

\noindent
\fbox{
\begin{minipage}[h]{6.4 in}
{\bf Framework~1} { (An NE-HPE framework for solving \eqref{eq:inc.p})}.
\begin{itemize}
\item[(0)] Let $z_0 \in Z$ and  $\sigma \in [0, 1]$ be given, and set $k=1$;
\item[(1)] choose $\lambda_k>0$ and find $(\tilde{z}_k, z_k, \varepsilon_k) \in Z \times Z \times \mathbb{R}_{+}$   such that
       \begin{align}
& r_k:= \frac{1}\lambda_k \nabla (dw)_{z_k}(z_{k-1})  \in T^{[\varepsilon_k]}(\tz_k), \label{breg-subpro} \\
& (dw)_{z_k}({\tz}_k) + \lambda_k\varepsilon_k \leq \sigma (dw)_{z_{k-1}}({\tz}_k); \label{breg-cond1}
\end{align}

\item[(2)] set $k\leftarrow k+1$ and go to step 1.
\end{itemize}
\noindent
{\bf end}
\end{minipage}
}

\vgap
\vgap

We now make some remarks about Framework~1. First, it does not specify
how to find $\lambda_k$ and $(\tilde{z}_k, z_k, \varepsilon_k)$ satisfying (\ref{breg-subpro}) and (\ref{breg-cond1}).
The particular 
scheme for computing $\lambda_k$ and $(\tilde{z}_k, z_k, \varepsilon_k)$ will depend on the instance of the framework under consideration
and the properties of the operator $T$.
Second, if $w$ is strongly convex on $Z$ and $\sigma= 0$, then (\ref{breg-cond1}) implies that $\varepsilon_k= 0$
and $z_k = \tilde z_k$ for every~$k$, and
hence that $r_k \in T(z_k)$ in view of \eqref{breg-subpro}.
Therefore, the HPE error conditions \eqref{breg-subpro}-\eqref{breg-cond1} can be viewed as a relaxation of an iteration of the exact non-Euclidean proximal point method,
namely,
\[
0 \in \frac{1}{\lambda_k} \nabla (dw)_{z_{k-1}}(z_{k})  +  T({z}_k).
\]

We observe that NE-HPE frameworks have already been
studied in \cite{Goncalves}, \cite{Kolossoski} and \cite{Solodov1}.
The approach presented in this section differs from these three papers as follows.
Assuming that $Z$ is an open convex set, $w$ is continuously differentiable on $Z$ and continuous on its closure,
\cite{Solodov1} studies a special case of the NE-HPE framework in which $\varepsilon_k=0$ for every $k$, and
presents results on convergence of sequences rather than iteration complexity.
Paper \cite{Kolossoski} deals with distance generating functions $w$ 
which do not necessarily satisfy conditions \eqref{eq:strongly-dist} and \eqref{eq:a1}, and as consequence,
obtains results which are more
limited in scope, i.e., only an ergodic convergence rate result is obtained for
operators with bounded feasible domains (or, more generally, for the case in which
the sequence generated by the HPE framwework is bounded).
Paper \cite{Goncalves} introduces the class of distance generating functions
$\mathcal{D}_Z(m,M)$ but only analyzes the behavior of a HPE framework
for solving inclusions whose operators are strongly monotone with respect to a fixed $w \in \mathcal{D}_Z(m,M)$  (see condition A1 in Section 2 of \cite{Goncalves}).
This section on the other hand assumes that $w \in \mathcal{D}_Z(m,M)$ but it does assume any strong monotonicity of
$T$ with respect to $w$.

Before presenting the main results about the the NE-HPE framework,
namely, Theorems \ref{th:alpha} and \ref{a2947} establishing its pointwise and ergodic iteration complexities, respectively,
and Propositions  \ref{lm:convergence} and \ref{prop:convergence} showing that $\{z_k\}$ and/or $\{\tz_k\}$ approach $T^{-1}(0)$ in terms
of the Bregman distance $(dw)$,  we first establish a few preliminary technical results.

\begin{lemma}\label{lema_desigualdades}
For every $k \geq 1$ and $z \in Z$, we have:
\begin{align}
(dw)_{z_{k-1}}(z) - (dw)_{z_k} (z) &= (dw)_{z_{k-1}}(\tilde{z}_k) - (dw)_{z_k} (\tilde{z}_k) + \lambda_k \langle r_k, \tilde{z}_k - z \rangle; \label{eq:des1} \\
(dw)_{z_{k-1}}(z) - (dw)_{z_k}(z) & \ge (1 - \sigma)(dw)_{z_{k-1}}(\tilde{z}_k) + \lambda_k (\langle r_k, \tilde{z}_k - z\rangle + \varepsilon_k); \label{eq:des2}  \\
(dw)_{z_{0}}(z) - (dw)_{z_k}(z) &\geq (1 - \sigma)\sum_{i=1}^k (dw)_{z_{i-1}}(\tilde{z}_i) +
\sum_{i=1}^k\lambda_i [ \langle r_i, \tilde{z}_i - z\rangle + \varepsilon_i ]. \label{eq:des3}
\end{align}
\end{lemma}
\proof
Using (\ref{equacao_taylor}) twice and the definition of $r_k$ in (\ref{breg-subpro}), we conclude that
\begin{align*}
 (dw)_{z_{k-1}}(z) - (dw)_{z_k}(z) &= (dw)_{z_{k-1}}(z_k) + \langle \nabla (dw)_{z_{k-1}}(z_k), z - z_k\rangle \\
& = (dw)_{z_{k-1}}(z_k) + \langle \nabla (dw)_{z_{k-1}}(z_k), \tilde{z}_k - z_k\rangle + \langle \nabla (dw)_{z_{k-1}}(z_k), z - \tilde{z}_k\rangle \\
& = (dw)_{z_{k-1}}(\tilde{z}_k) - (dw)_{z_k}(\tilde{z}_k) + \langle \nabla (dw)_{z_{k-1}}(z_k), z - \tilde{z}_k\rangle \\
& = (dw)_{z_{k-1}}(\tilde{z}_k) - (dw)_{z_k}(\tilde{z}_k) + \lambda_k\langle r_k, \tilde{z}_k - z\rangle,
\end{align*}
and hence that \eqref{eq:des1} holds.
Inequality \eqref{eq:des2} follows immediately from \eqref{eq:des1} and (\ref{breg-cond1}).
Moreover, \eqref{eq:des3} follows by adding \eqref{eq:des2} from $k=1$ to $k=k$.
%
\endproof

\begin{proposition}\label{proposition_desigualdades}
For every $k \geq 1$ and $z^* \in T^{-1}(0)$, we have
\begin{eqnarray}\label{eq:monotone}
(dw)_{z_{k-1}}(z^*) - (dw)_{z_k}(z^*) - (1 - \sigma)(dw)_{z_{k-1}}(\tilde{z}_k)
\ge \lambda_k \left[ \langle r_k, \tilde{z}_k - z^\ast \rangle + \varepsilon_k  \right]\geq 0.
\end{eqnarray}
As a consequence,
the following statements hold:
\begin{itemize}
\item[(a)]
$\{(dw)_{z_k}(z^*)\}$ is non-increasing;
\item[(b)]
$\lim_{k\to \infty} \lam_k \left[\langle r_k, \tilde{z}_k - z^\ast \rangle + \varepsilon_k \right] = 0$;
\item[(c)]
$ (1-\sigma) \sum_{i=1}^k (dw)_{z_{i-1}}(\tz_i) \le (dw)_{z_0}(z^*)$.
\end{itemize}
%
\end{proposition}
\proof
Let $z^* \in T^{-1}(0)$ be given. The first inequality in \eqref{eq:monotone} follows from \eqref{eq:des2} with $z=z^*$ and
the last inequality in \eqref{eq:monotone} follows from the fact that $0 \in T(z^*)$ and $r_k \in T^{[\varepsilon_k]}(\tz_k)$,
and the definition of $T^{[\varepsilon]}(\cdot)$.
Finally, statements (a) and (b) follow immediately from \eqref{eq:monotone} while (c) follows by adding \eqref{eq:monotone} over $i=1,\ldots,k$
and using the fact that
$(dw)_{z_k}(z^*) \ge 0$ for every $k$.
\endproof

For the purpose of stating the convergence rate results below, define
\beq \label{eq:dw0-def}
(dw)_0 := \inf \{ (dw)_{z_0}(z^*) : z^* \in T^{-1}(0)\}.
\eeq

\begin{lemma} \label{lm:breg.bas1}
For every $i \ge 1$, define
\beq \label{eq:breg-tau-1}
\theta_i := \max \left \{ \frac{\lambda_i^{2} \|r_i\|_* ^2}{\tau^2 (1+ \sqrt{\sigma})^2} ,
\frac{\lambda_i^{} \varepsilon_i}{\sigma} \right\} \ \ \mbox{ where} \ \ \tau := \frac{\sqrt{2}M}{\sqrt{m}}.
\eeq
Then, $(1-\sigma) \sum_{i=1}^k \theta_i \le dw_0$.
\end{lemma}
\begin{proof}
For every $i \ge 1$, it follows from \eqref{breg-subpro}, \eqref{grad-d}, \eqref{eq:789},
\eqref{breg-cond1}, the triangle inequality for norms and the above definition of $\tau$, that
\begin{align*}
\lambda_i \|r_i\|_* &= \| \nabla (dw)_{z_{i}}(z_{i-1}) \|_* =  \|\nabla (dw)_{z_{i-1}}(\tz_i) - \nabla (dw)_{z_{i}}(\tz_i) \|_* \\
&\le \|\nabla (dw)_{z_{i-1}}(\tz_i)\|_* + \|\nabla (dw)_{z_{i}}(\tz_i) \|_* \le \tau \left[ (dw)_{z_{i-1}}(\tz_i)^{1/2} + (dw)_{z_{i}}(\tz_i)^{1/2} \right] \\
& \le
\tau ( 1 + \sqrt{\sigma} ) (dw)_{z_{i-1}}(\tz_i)^{1/2}.
\end{align*}
The last inequality, \eqref{breg-cond1} and the definition of $\theta_i$ then imply that
$\theta_i \le (dw)_{z_{i-1}}(\tz_i)$ for every $ i \ge 1$.
Hence, if $z^* \in T^{-1}(0)$, it follows that
\[
(1-\sigma) \sum_{i=1}^k \theta_i \le (1-\sigma) \sum_{i=1}^k (dw)_{z_{i-1}}(\tz_i) 
\le (dw)_{z_0}(z^*)
\]
where  the last inequality follows from Proposition  \ref{proposition_desigualdades}(c).
The lemma now follows from the latter relation and the definition of $(dw)_0$ in \eqref{eq:dw0-def}.
\end{proof}

\begin{lemma}
  \label{lm:breg.alpha-1}
Let $(dw)_0$ be as in \eqref{eq:dw0-def} and $\tau$ be as in \eqref{eq:breg-tau-1}, and assume that $\sigma<1$.
Then, for every
  $\alpha\in \R$ and every $k \ge 1$, there exists an $i\leq k$ such that
  \begin{equation}
    \label{v_ieps_i-bound-a-1}
    \norm{r_i}_* \leq \tau(1+\sqrt{\sigma})
    \sqrt{\frac{(dw)_0}{1-\sigma}
    \, \left( \frac{\lambda_i^{\alpha-2}}
   { \sum_{j=1}^k \lambda_j^\alpha}\right)} ,
    \quad \quad \quad
    \varepsilon_i\leq
     \frac{ \sigma (dw)_0} {1-\sigma}
    \left( \frac{\lambda_i^{\alpha-1}}{\sum_{j=1}^k \lambda_j^\alpha}
    \right).
\end{equation}
\end{lemma}
\begin{proof}
It follows from Lemma \ref{lm:breg.bas1} that
\[
\frac{ (dw)_0 }{1-\sigma} \ge  \sum_{i=1}^k \theta_i  =
\sum_{i=1}^k \frac {\theta_i}{\lambda_i^{\alpha}}  \lambda_i^\alpha
\ge \left( \min_{i=1,\ldots,k}  \frac {\theta_i}{\lambda_i^{\alpha}} \right)
\left( \sum_{i=1}^k \lambda_i^\alpha \right)
\]
which, in view of the definition of $\theta_i$ in \eqref{eq:breg-tau-1},
can be easily seen to be equivalent to the conclusion of the lemma.
\end{proof}

The following pointwise convergence rate result describes the convergence rate of the sequence $\{(r_k,\varepsilon_k)\}$ of residual pairs
associated to the sequence $\{\tz_k\}$. Note that its convergence rate bounds are derived on the best residual pair
among $(r_i,\varepsilon_i)$ for $i=1,\ldots,k$ rather than on the last residual pair $(r_k,\varepsilon_k)$.

\begin{theorem} {\bf (Pointwise convergence)} \label{th:alpha}
Let $(dw)_0$ be as in \eqref{eq:dw0-def} and $\tau$ be as in \eqref{eq:breg-tau-1}, and assume that $\sigma<1$.  Then, the following statements hold:
  \begin{itemize}
\item[(a)]  if $\underline{\lambda}:=\inf\lambda_k>0$,
then for every $k\in\mathbb{N}$ there exists $i\leq k$ such that
  \[
  \norm{r_i}_* \leq  \tau(1+\sqrt{\sigma}) \sqrt{\frac{(dw)_0} {1-\sigma}\;
\left(\frac{\underline \lambda ^{-1}}{\sum_{j=1}^k\lambda_j}\right)}
\leq
\frac{\tau(1+\sqrt{\sigma})}{\underline\lambda\sqrt k} \sqrt{\frac{(dw)_0}{1-\sigma}}
\]
\[
  \varepsilon_i\leq
\frac{\sigma (dw)_0}{1-\sigma}\frac{1}{\sum_{i=1}^k\lambda_i}
\leq
\frac{\sigma (dw)_0}{(1-\sigma)\underline
    \lambda k};
  \]
  \item[(b)]
  for every $k \in\mathbb{N}$, there exists an index $i\leq k$ such that
  \begin{equation}
 \label{v_ieps_i-bound-b-1}
  \norm{r_i}_*\leq \tau(1+\sigma)
  \sqrt{\frac{(dw)_0}{1-\sigma}
\left(\frac{1}{ \sum_{j=1}^k \lambda_j^2}\right)
}
  ,
  \quad \quad \quad
  \varepsilon_i\leq\frac{\sigma (dw)_0 \lambda_i}{(1-\sigma)\sum_{j=1}^k \lambda_j^2}.
\end{equation}
\end{itemize}
\end{theorem}
\begin{proof}
  Statements (a) (resp., (b)) follows from Lemma~\ref{lm:breg.alpha-1} with $\alpha=1$
  (resp., $\alpha=2$).
\end{proof}

From now on, we focus on the ergodic convergence rate of the NE-HPE framework.
For $k \geq 1$, define $\Lambda_{k} := \sum_{i=1}^k \lambda_i$ and the ergodic sequences
\begin{equation}\label{SeqErg}
\tilde z^a_{k} = \frac{1}{\Lambda_{k}}\sum_{i=1}^k \lambda_i \tilde z_i, \quad
r^a_{k} := \frac{1}{\Lambda_{k}}\sum_{i=1}^k \lambda_i r_i, \quad
\varepsilon^a_{k} := \frac{1}{\Lambda_{k}} \sum_{i=1}^k \lambda_i \left( \varepsilon_i + \inner{r_i}{\tilde z_i -\tilde z^a_{k}} \right).
\end{equation}

The following ergodic convergence result describes the association between the ergodic iterate $\tz_k^a$ and
the residual pair $(r_k^a,\varepsilon_k^a)$,
and gives a convergence rate bound on the latter residual pair.

\begin{theorem}{\bf(Ergodic convergence)} \label{a2947}
Let $(dw)_0$ be as in \eqref{eq:dw0-def} and $\tau$ be as in \eqref{eq:breg-tau-1}. Then, for every $k\geq 1$, we have
\[
\varepsilon_k^a \ge 0, \quad r^a_k \in T^{[\varepsilon_k^a]}(\tz^a_k)
\]
and
\begin{align*}
 \|r_k^a\|_* \le \frac{2 \tau \sqrt{(dw)_0}}{\Lambda_k}, \quad
 \varepsilon^a_{k} \leq \left(\frac{3M}{m}  \right)
\frac{ 2 (dw)_{0} + \rho_k}{\Lambda_k}
\end{align*}
where
\beq \label{eq:def-rhok}
\rho_k := \displaystyle\max_{i=1,\ldots,k}(dw)_{z_{i}}(\tilde z_{i}).
\eeq
Moreover, the sequence $\{\rho_k\}$ is bounded under either
one of the following situations:
\begin{itemize}
\item[(a)]
$\sigma<1$, in which case
\begin{equation} \label{def:tauk}
\rho_k \le \frac{\sigma (dw)_0}{1-\sigma};
\end{equation}
\item[(b)]
$\Dom T$ is bounded, in which case
\beq
\rho_k \le \frac{2M}{m} [ (dw)_0 + D]  \label{def:tauk1}
\eeq
where $D := \sup \{ \min\{ (dw)_y(y'), (dw)_{y'}(y) \} :  y,y' \in \Dom T\}$ is the diameter
of $\Dom T$ with respect to $dw$.
\end{itemize}
\end{theorem}

%
%
\proof
The inequality $\varepsilon_k^a \ge 0$ and the inclusion $r^a_k \in T^{[\varepsilon_k^a]}(\tz^a_k)$ follows from \eqref{SeqErg} and
 the transportation formula (see  \cite[Theorem~2.3]{Burachik1}).  
Now, let $z^* \in T^{-1}(0)$ be given.
Using \eqref{grad-d}, \eqref{breg-subpro} and \eqref{SeqErg}, we easily see that
\[
\Lambda_k r^a_k = \sum_{i=1}^k \lam_i r_i =  \sum_{i=1}^k \nabla (dw)_{z_i}(z_{i-1}) = \nabla (dw)_{z_k}(z^*) - \nabla (dw)_{z_0}(z^*).
\]
Hence, in view of Proposition \ref{proposition_desigualdades}(a),  and relations \eqref{eq:789} and \eqref{eq:breg-tau-1}, we have
\begin{align*}
\Lambda_k \norm{r^a_k}_* &=  \norm{\nabla (dw)_{z_0}(z^*)}_* + \norm{ \nabla (dw)_{z_k}(z^*)}_* \\
&\le \tau [ (dw)_{z_0}(z^*)^{1/2} + (dw)_{z_k}(z^*)^{1/2} ] \le 2 \tau (dw)_{z_0}(z^*)^{1/2}.
\end{align*}
This inequality together with definition of $(dw)_0$ clearly imply the bound on $\|r_k\|_*$.
We now establish the bound on $\varepsilon_k^a$.
Using  inequality \eqref{eq:des3} with $z=\tilde z^a_{k}$, noting  \eqref{SeqErg}, and using the fact that $(dw)_{z_0}(\cdot)$ is convex and
 $\sigma \le 1$, we conclude that
\[
\Lambda_k \varepsilon_k^a = \sum_{i=1}^k\lambda_i (\langle r_i, \tilde{z}_i - \tz^a_k\rangle + \varepsilon_i)
\le  (dw)_{z_{0}}(\tilde z^a_{k}) \le \max_{i=1,\ldots,k} (dw)_{z_0}(\tz_i).
\]
On the other hand, \eqref{eq:56} with $l=3$ implies that for every $i \ge 1$  and $z^* \in T^{-1}(0)$,
\begin{align*}
(dw)_{z_{0}}(\tilde z_{i}) &\le
\frac{3M}{m} \left[ (dw)_{z_{i}}(\tilde z_{i})+(dw)_{z_{i}}(z^*)+(dw)_{z_{0}}( z^*) \right] \\
& \le \frac{3M}{m} \left[ (dw)_{z_{i}}(\tilde z_{i})+2 (dw)_{z_{0}}( z^*) \right]
\end{align*}
where the last inequality is due to  Proposition \ref{proposition_desigualdades}(a).
Combining the above two relations and using the definitions of $\rho_k$ and $(dw)_0$, we then conclude that
the bound on $\varepsilon_k^a$ holds.

We now establish the bounds on $\rho_k$ under either one of the conditions  (a) or (b).
First, if $\sigma<1$, then it follows
from \eqref{breg-cond1} and Proposition \ref{proposition_desigualdades} that

\[
 (dw)_{z_{i}} (\tz_i) \le \sigma (dw)_{z_{i-1}} (\tz_i) \le \frac{\sigma}{1-\sigma} (dw)_{z_{i-1}}(z^*) \le \frac{\sigma}{1-\sigma} (dw)_{z_{0}}(z^*)
\]
for every
$i \ge 1$ and $z^* \in T^{-1}(0)$. Noting \eqref{eq:dw0-def} and \eqref{eq:def-rhok}, we then conclude that \eqref{def:tauk} holds.
Assume now that $\Dom T$ is bounded. Using \eqref{eq:56} with $l=2$ and Proposition \ref{proposition_desigualdades}(a), and noting the
definition of $D$ in (b), we conclude that
\[
(dw)_{z_{i}} (\tz_i) \le \frac{2M}m \left[ (dw)_{z_{i}}(z^*) + \min\{ (dw)_{\tz_i}(z^*) , (dw)_{z^*}(\tz_i) \} \right]
\le \frac{2M}m \left[ (dw)_{z_{0}}(z^*) + D \right]
\]
for every $i \ge 1$ and $z^* \in T^{-1}(0)$. Hence, noting \eqref{eq:dw0-def} and \eqref{eq:def-rhok}, we conclude that \eqref{def:tauk1} holds.
\endproof

{\color{black}{In the remaining part of this subsection, we state some results about the sequence generated by
an instance of the NE-HPE framework. We assume from now on that
such instance generates an infinite sequence of iterates, i.e., the instance does not
terminate in a finite number of steps and no termination criterion is checked.
Since we are not assuming that the distance generating function $w$ is nondegenerate on $Z$, it is not possible to establish convergence
of the sequence $\{z_k\}$ generated by the NE-HPE framework to a solution of \eqref{eq:inc.p}. However, under some mild assumptions,
it is possible to establish that
$\{z_k\}$ approaches a point $\tz \in T^{-1}(0)$ if the proximity measure used is the actual Bregman distance.}}

\begin{proposition}\label{lm:convergence}
Assume that for some infinite index set $\cK$ and some $\tz \in \cZ$, we have
\beq \label{eq:limmm}
\lim_{k \to \cK} (r_k,\varepsilon_k)=(0,0),
\quad \lim_{k \to \cK} \tz_k = \tz.
\eeq
Then, $\tz \in T^{-1}(0) \subset Z$. If, in addition, $\lim_{k \in \cK} (dw)_{z_k}(\tz_k) =0$, then
$\lim_{k \to \infty} (dw)_{z_k}(\tz) =0$.
\end{proposition}

\begin{proof}
Using the two limits in \eqref{eq:limmm}, and
the fact that every maximal monotone operator is closed and
$r_k \in T^{\varepsilon_k}(\tz_k)$ for every $k \in \cK$, we conclude that
$0 \in T^{0}(\tz)=T(\tz)$. This conclusion together with Assumption A0 then imply
that the first assertion of the proposition holds and that
$\{(dw)_{z_k}(\tz)\}$ is non-increasing in view of Proposition  \ref{proposition_desigualdades}(a).
To show the second assertion, assume that $\lim_{k \in \cK} (dw)_{z_k}(\tz_k) =0$.
Since Lemma \ref{basicassu} with $l=2$ implies
\[
(dw)_{z_{k}}(\tz) \le \frac{2M}{m} \left[ (dw)_{z_{k}}(\tz_k) +  (dw)_{\tz_k}(\tz) \right],
\]
and the second limit in  \eqref{eq:limmm} clearly implies that $\lim_{k \in \cK} (dw)_{\tz_k}(\tz) =0$,
we then conclude  that $\lim_{k\in \cK} (dw)_{z_{k}}(\tz)=0$.
Clearly, since $\{ (dw)_{z_k}(\tilde{z}) \}$ is non-increasing, we have that $\lim_{k\to \infty} (dw)_{z_{k}}(\tz)=0$, and hence that
the second assertion holds. 
\end{proof}

\begin{proposition}\label{prop:convergence}
Assume that $\sigma<1$, $\sum_{i=1}^\infty \lambda_k^2 = \infty$ and $\{\tz_k\}$ is bounded. Then, there exists
$\tz \in T^{-1}(0) \subset Z$  such that
\beq \label{eq:limm'}
\lim_{k \to \infty} (dw)_{z_k}(\tz) = \lim_{k \to \infty} (dw)_{\tz_k}(\tz) = 0.
\eeq
\end{proposition}
\begin{proof}
The assumption that $\sigma<1$ and $\sum_{i=1}^\infty \lambda_k^2 = \infty$ together with Theorem \ref{th:alpha}(b) imply that
there exists subsequence $\{(r_k,\varepsilon_k)\}_{k \in {\cal K}}$ converging to zero.
Since $\{\tz_k\}_{k \in {\cal K}}$ is bounded, we may assume without loss of generality (by passing to a subsequence if necessary)
that $\{\tz_k\}_{k \in {\cal K}}$ converges to some $\tz \in \cZ$.
Hence, by the first part of Proposition \ref{lm:convergence}, we conclude that $\tz \in T^{-1}(0) \subset Z$.
Thus, Proposition \ref{proposition_desigualdades}(c)  with $z^*=\tz$ and the assumption that $\sigma<1$ imply that
$\lim_{k\to\infty} (dw)_{z_{k-1}}(\tz_k) =0$, and hence that
$\lim_{k\to\infty} (dw)_{z_{k}}(\tz_k) =0$ in view of \eqref{breg-cond1}.
This conclusion together with our previous conclusion that \eqref{eq:limmm} holds and
the second part of Proposition \ref{lm:convergence} then imply that $\lim_{k \to \infty} (dw)_{z_k}(\tz) = 0$.
The latter conclusion together with the fact that  $\lim_{k\to\infty} (dw)_{z_{k}}(\tz_k) =0$
and  Lemma \ref{basicassu} with $l=2$ easily imply that $\lim_{k \to \infty} (dw)_{\tz_k}(\tz) = 0$.
\end{proof}

\vgap

Clearly, if $w$ is a nondegenerate distance generating function, then the results above give sufficient conditions for the
sequences $\{z_k\}$ and $\{\tz_k\}$ to converge to some  $\tz \in T^{-1}(0)$.

%
%

\section{The relaxed Peaceman-Rachford splitting method}\label{sec:PRmethod}

{\color{black}{This section derives convergence rate bounds for the relaxed Peaceman-Rachford (PR) splitting method for solving
the monotone inclusion \eqref{eq:mainincl} under the assumption that $A$ and $B$ are maximal $\beta$-strongly monotone operators
for any $\beta \ge 0$. More specifically, its pointwise iteration-complexity is obtained in Theorem \ref{thm:pointwise2} and
its ergodic iteration-complexity is derived in Theorem \ref{thm:main}. These results are obtained as by-products of
the corresponding ones (i.e, Theorem \ref{th:alpha} and Theorem \ref{a2947}) in Subsection \ref{subsec:HPE} and the fact that the relaxed Peaceman-Rachford (PR) splitting method
can be viewed as a special instance of the NE-HPE framework.}}

Throughout this section, we assume that {\color{black}{$\cX$ a finite-dimensional
real vector space with
inner product and associated inner product norm denoted by $\inner{\cdot}{\cdot}_{\cX}$ and $\| \cdot \|_{\cX}$}},
respectively.
For a given $\beta \geq 0$, an operator $T: \cX \tos \cX$ is said to be $\beta$-strongly monotone if
\[
\inner {w-w'}{x-x'}_{\cX} \ge \beta \|x-x'\|_\cX^2 \quad \forall (x,w), (x',w') \in {\rm Gr}(T).
\]
In what follows, we refer to monotone operators as $0$-strongly monotone
operators. This terminology has the benefit of allowing us to treat
both the monotone and strongly monotone case simultaneously.

Throughout this section, we consider the monotone inclusion \eqref{eq:mainincl}
where $A,B : \cX \tos \cX$ satisfy the following assumptions:
\begin{itemize}
\item[B0)]
for some $\beta \ge 0$, $A$ and $B$ are maximal $\beta$-strongly monotone operators;
\item[B1)]
the solution set $(A+B)^{-1}(0)$ is non-empty.
\end{itemize}


{\color{black}{
We start by
observing that \eqref{eq:mainincl}  is equivalent to solving the following augmented system of  inclusions/equation
\begin{align*}
0 &\in \gamma A(u) + u - x, \\
0 &\in \gamma B(v) + x - v, \\
0 &= u-v
\end{align*}
where $\gamma > 0$ is an arbitrary scalar.
Another way of writing the above system is as 
\begin{align*}
0 &\in \gamma A(u) + u - x, \\
0 &\in \gamma B(v) + v + x - 2u,  \\
0 &= u-v.
\end{align*}
}}
Note that the first and second {\color{black}{inclusions}} are equivalent to
\begin{align}\label{eq:uv}
u = u(x) := J_{\gamma A}(x), \quad
v=v(x):=J_{\gamma B}(2u-x) = J_{\gamma B}(2J_{\gamma A}(x)-x)
\end{align}
so that the third equation reduces to
\[
0 = u(x) -v(x) = J_{\gamma A}(x) -  J_{\gamma B}(2J_{\gamma A}(x)-x).
\]
%
%
%
%
%
The Douglas-Rachford (DR) splitting method is the iterative procedure
$x_k = x_{k-1} + v(x_{k-1})-u (x_{k-1})$, $k\ge 1$, started from some $x_0 \in \cX$.  It is known that the DR
splitting method is an exact proximal point method for some maximal monotone operator \cite{Eckstein,Facchinei}. 
Hence, convergence of its sequence of iterates is guaranteed.

This section is concerned with a natural generalization of the DR splitting  method, namely,
the relaxed Peaceman-Rachford (PR) splitting method with relaxation parameter $\theta >0$, which iterates as
\beq \label{eq:DR-iter-theta}
(u_k,v_k) := (u(x_{k-1}),v(x_{k-1})) \quad
x_k = x_k^\theta  := x_{k-1} + \theta (v_k-u_k) \quad \forall k \ge 1.
\eeq
We now make a few remarks about the above method.
First, it reduces to the DR splitting method when $\theta=1$, and to the PR splitting method
when $\theta=2$. Second, it reduces to \eqref{eq:seqqq} when $\gamma=1$ but it is not more general
than \eqref{eq:seqqq} since   \eqref{eq:DR-iter-theta} is equivalent to \eqref{eq:seqqq} with $(A,B)=(\gamma A, \gamma B)$.
Third, as presented in \eqref{eq:DR-iter-theta}, it can be viewed as an iterative process in the $(u,v,x)$-space
rather than only in the $x$-space as suggested by \eqref{eq:seqqq}.

Our analysis of the relaxed DR splitting method is based on further exploring the last remark above, i.e.,
viewing it as an iterative method in the $(u,v,x)$-space.
We start by introducing an inclusion which plays an important role in our analysis.
For a fixed $\tilde \theta>0$ and $\gamma>0$, consider the inclusion
\beq \label{eq:inclu-new}
0 \in (\cL_{\ttheta}+\gamma C)(z)
\eeq
where $\cL_{\tilde{\theta}} : \cX \times \cX \times \cX \to \cX \times \cX \times \cX$ is the linear map
defined as
\beq \label{eq:def-Ltheta}
\cL_{\tilde{\theta}}(z) = \cL_{\tilde{\theta}}(u,v,x) := \left[
\begin{array}{ccc}
(1-\ttheta)I & \ttheta I & -I \\ (\ttheta-2)I & (1-\ttheta)I & I \\ I & - I & 0
\end{array}
\right] \left( \begin{array}{c} u \\ v \\ x \end{array} \right)
\eeq
and $C: \cX \times \cX \times \cX \tos \cX \times \cX \times \cX$ is the maximal monotone operator defined as
\beq \label{eq:def-calC}
C(z) = C(u,v,x) := A(u) \times B(v) \times \{0\}.
\eeq
It is easy to verify that the inclusion \eqref{eq:inclu-new} is equivalent to the two systems of inclusions/equation following conditions B0 and B1.
Hence, it suffices to solve \eqref{eq:inclu-new} in order to solve \eqref{eq:mainincl}.
The following simple but useful result explicitly show the relationship between the solution sets of \eqref{eq:inclu-new} and \eqref{eq:mainincl}.

\begin{lemma} \label{lm:sol}
For any $\tilde \theta>0$,  
the solution set
$(\cL_{\tilde \theta}+\gamma C)^{-1}(0)$ is given by
\begin{align*}
(\cL_{\tilde \theta}+\gamma C)^{-1}(0) &=
\{(u^*,u^*,x^*) :  \gamma^{-1}(x^*-u^*) \in A(u^*) \cap (-B(u^*)) \} \\
&=
\{ (u^*,u^*,u^*+ \gamma a^*) :a^* \in A(u^*), -a^* \in B(u^*) \}.
\end{align*}
As a consequence, if $z^*=(u^*,u^*,x^*) \in (\cL_{\tilde \theta}+\gamma C)^{-1}(0)$, then
$u^* \in (A+B)^{-1}(0)$ and $u^*=J_{\gamma A}(x^*)$.
\end{lemma}

\begin{proof}
The conclusion of the lemma follows immediately from
 the definitions of $\cL_{\tilde \theta}$ and {\color{black}{$C$ in \eqref{eq:def-Ltheta}
and \eqref{eq:def-calC}}}, respectively,
and
some simple algebraic manipulations.
\end{proof}

The key idea of our analysis is to show that the relaxed PR splitting method is actually a special instance of
the NE-HPE framework for solving inclusion \eqref{eq:inclu-new} and then use the results discussed in Subsection \ref{subsec:HPE} to derive
convergence and iteration-complexity results for it.
With this goal in mind, the next result gives a sufficient condition for
\eqref{eq:inclu-new} to be a maximal monotone inclusion.

{\color{black}{
\begin{proposition}\label{prop:etarelax}
Assume that $A,B : \cX \tos \cX$ satisfy B0 and let $\tilde \theta >0$ be given.  Then,  
\begin{itemize}
\item[(a)] for every $z = (u,v,x) \in \cX \times \cX \times \cX$, $z'=(u',v',x') \in \cX \times \cX \times \cX$, $r 
 \in (\cL_{\tilde{\theta}}+\gamma C)(z)$ and 
$r' 
\in (\cL_{\tilde{\theta}}+\gamma C)(z')$, 
we have
\begin{align} \label{Ltildethetainequality}
\langle \cL_{\tilde{\theta}}(z - z'), z - z') &= (1 - \tilde{\theta}) \| (u-u') - (v-v')  \|_{\cX}^2 \\
\langle r-r', z-z' \rangle 
\label{Ltildethetainequality'}
 &\ge  (1 - \tilde{\theta}) \| (u-u') - (v-v')  \|_{\cX}^2 + \gamma \beta (\|u-u' \|^2_{\cX} + \| v - v'\|^2_{\cX} );
\end{align}
\item[(b)] $\cL_\ttheta + \gamma C$ is maximal monotone  whenever
$\tilde{\theta} \in (0, \theta_0]$ where
\beq \label{eq:theta0}
\theta_0 := 1 + \frac{\gamma \beta}2.
\eeq
\end{itemize}
\end{proposition}
\begin{proof}
(a) Identity \eqref{Ltildethetainequality} follows from the definition of $\cL_{\tilde{\theta}}$ in \eqref{eq:def-Ltheta}.
To show inequality \eqref{Ltildethetainequality'}, assume that $r \in  (\cL_{\tilde{\theta}}+\gamma C)(z)$ and 
$r'  \in (\cL_{\tilde{\theta}}+\gamma C)(z')$. Then, $r = \cL_{\tilde{\theta}}(z) + \gamma c$ and
$r' = \cL_{\tilde{\theta}}(z) + \gamma c'$ for some $c \in C(z)$ and $c' \in C(z')$.
Using the definition of $C$ and assumption B0, we easily see that
\[
\inner{c'-c}{z'-z} \ge \beta (\|u-u' \|^2_{\cX} + \| v - v'\|^2_{\cX} ),
\]
which together with \eqref{Ltildethetainequality}, and the fact that  $r = \cL_{\tilde{\theta}}(z) + \gamma c$ and
$r' = \cL_{\tilde{\theta}}(z) + \gamma c'$,  imply \eqref{Ltildethetainequality'}.

(b) Monotonicity of $\cL_{\tilde{\theta}} + \gamma C$ is due to the fact that the right hand side of
\eqref{Ltildethetainequality'} is  nonnegative for every $(u,v), (u',v') \in \cX \times \cX$ whenever
$\tilde{\theta} \in (0, \theta_0]$.
To show $\cL_{\tilde{\theta}} + \gamma C$ is maximal monotone, 
write $\cL_\ttheta + \gamma C = (\cL_\ttheta + \gamma \bar{C}) + \gamma(C-\bar{C})$ where $\bar{C}:=\beta(I,I,0)$. 
As a consequence of (a) with $(A,B)=\beta (I,I)$ and the definition of $C$, we conclude that
$\cL_\ttheta + \gamma \bar{C}$ is a monotone linear operator for every $\ttheta \in (0, \theta_0]$.
Moreover, Assumption B0 easily implies that $\gamma(C-\bar{C})$ is maximal monotone.  
The statement now follows by noting that the sum of a monotone linear map and a maximal monotone operator is a maximal monotone operator \cite{Bauschke,Rockafellar2}. 
\end{proof}
}}

Note that $\theta_0$ in \eqref{eq:theta0} depends on $\gamma$ and $\beta$ and that $\theta_0 = 1$ when $\beta=0$.




{\color{black}{The following technical result states some useful identities and inclusions needed to analyze the
the sequence generated by the relaxed PR splitting method.}}




{\color{black}{
\begin{lemma} \label{lm:DR.1}
For a given $x_{k-1} \in \cX$ and $\tilde \theta>0$, define
\begin{equation} \label{eq:zk-theta}
\tilde x_k = x_k^{\tilde \theta} := x_{k-1} + \ttheta (v_k-u _k), \quad \tilde z_k = z_k^{\tilde \theta} := (u_k,v_k,\tx_k)
\end{equation}
where $u_k, v_k$ 
are as in \eqref{eq:DR-iter-theta}, and set
\beq \label{eq:def-ak-bk}
a_k := \frac{1}{\gamma} (x_{k-1}-u_k), \quad b_k := \frac{1}\gamma (2 u_k - v_k - x_{k-1}).
\eeq
Then, we have:
\begin{align}
 - (1-\ttheta) u_k - \ttheta v_k + \tx_k &= \gamma a_k \in \gamma A(u_k), \label{ak}\\
 (2-\ttheta) u_k -  (1-\ttheta)v_k - \tx_k &= \gamma b_k \in \gamma B(v_k). \label{bk}
\end{align}
As a consequence, we have
\begin{align}
 \label{eq:DR-incl-aug'}
u_k-v_k &= \gamma(a_k+b_k)   \in \gamma A(u_k) + \gamma B(v_k) \\
(0,0, u_k-v_k) &= \cL_\ttheta(\tz_k) + \gamma c_k \in (\cL_\ttheta+ \gamma C) (\tz_k) 
 \label{eq:DR-incl-aug}
\end{align}
where
\beq \label{eq:def-ck}
c_k := (a_k,b_k,0).
\eeq
\end{lemma}
\begin{proof}
Using the definition of $(u(\cdot),v(\cdot))$ in \eqref{eq:uv}, the definition of $(u_k,v_k,x_k^\ttheta)$ in
\eqref{eq:DR-iter-theta}, and the definitions of $a_k$ and $b_k$ in \eqref{eq:def-ak-bk}, we easily see that
\eqref{ak} and \eqref{bk} hold. {\color{black}{The equality and the inclusion in \eqref{eq:DR-incl-aug'} follow
by adding \eqref{ak} and \eqref{bk}.}}
Clearly, \eqref{eq:DR-incl-aug} follows as an immediate consequence of
\eqref{ak} and \eqref{bk}, definitions \eqref{eq:def-Ltheta} and \eqref{eq:def-calC}, and
the definition of $c_k$.
\end{proof}

The following result shows that the relaxed PR splitting method with $\theta \in (0,2\theta_0]$ can be viewed
as an inexact instance of the NE-HPE framework for solving \eqref{eq:inclu-new} where from now on we assume that
\begin{equation} \label{eq:tildetheta}
\tilde \theta :=
\min \{ \theta, \theta_0 \}.
\end{equation}

\begin{proposition} \label{pr:DR.2}
Consider the
(degenerate) distance generating function given by
\beq \label{eq:DR-w-def}
w(z) = w(u,v,x) = \frac{\norm{x}^2_\cX}{2\theta} \quad \forall z=(u,v,x) \in \cX \times \cX \times \cX
\eeq
and the sequence $\{z_k=(u_k,v_k,x_k)\}$ generated according to the
relaxed PR splitting method \eqref{eq:DR-iter-theta} with any $\theta>0$. Also, define the sequences
$\{\varepsilon_k\}$, $\{\lambda_k\}$ and $\{r_k\}$ as
\beq \label{eq:def-rk}
\varepsilon_k:=0, \quad \lambda_k:=1, \quad r_k := \nabla (dw)_{z_{k}}(z_{k-1}) \ \ \ \forall k \ge 1,
\eeq
and the sequence $\{\tilde z_k = (u_k, v_k,\tilde x_k) \}$
as in \eqref{eq:zk-theta} with $\tilde \theta$ given by \eqref{eq:tildetheta}.
Then, for every $k \ge 1$, we have:
\begin{itemize}
\item[(a)]
$r_k = (0,0,x_{k-1}-x_k)/\theta = (0,0,u_k- v_k) = \gamma (0,0, a_k + b_k)$;
\item[(b)]
$(\lambda_k,z_{k-1})$ and $(z_k,\tz_k,\varepsilon_k)$ satisfy
\eqref{breg-subpro} with $T=\cL_{\tilde \theta}+\gamma C$, i.e., $
r_k \in (\cL_{\tilde \theta}+\gamma C) (\tz_k)
$;
\item[(c)]
$(\lambda_k,z_{k-1})$ and $(z_k,\tz_k,\varepsilon_k)$ satisfy \eqref{breg-cond1} with $\sigma = (\theta/\tilde \theta-1)^2$
and $w$ as in \eqref{eq:DR-w-def}.
\end{itemize}
As a consequence, the relaxed PR splitting method with $\theta \in (0,2\theta_0)$ (resp., $\theta=2\theta_0$) is an NE-HPE instance
with respect to the monotone inclusion $0 \in (\cL_{\tilde \theta}+\gamma C)(z)$ in which 
$\sigma<1$ (resp., $\sigma=1$), $\varepsilon_k = 0$ and $\lambda_k=1$ for every $k$.
\end{proposition}

\begin{proof}
(a) The first identity in (a) follows from \eqref{eq:DR-w-def} and the definition of $r_k$ in \eqref{eq:def-rk}.
The second and third equalities in (a) are due to the second identity in \eqref{eq:DR-iter-theta} and relation \eqref{eq:DR-incl-aug'},
respectively.

(b) This statement follows from (a) and \eqref{eq:DR-incl-aug}.

(c) Using the second identity in \eqref{eq:DR-iter-theta}, relation \eqref{eq:DR-w-def}
and the definition of $\tx_k$ in \eqref{eq:zk-theta}, we conclude that for any $\theta \in (0,2\theta_0]$,
\[
(dw)_{z_k}(\tz_k) = \frac{\|\tx_k-x_k\|^2_{\cX}}{2\theta} = \left(\frac{\theta}{\tilde\theta}-1\right)^2 \frac{\|\tx_k-x_{k-1}\|^2_{\cX}}{2\theta} =  \left(\frac{\theta}{\tilde\theta}-1\right)^2 (dw)_{z_{k-1}}(\tz_k)
\]
and hence that \eqref{breg-cond1} is satisfied with $\sigma=(1-\theta/\tilde\theta)^2$.  

The last conclusion follows from
statements (b) and (c), and 
Proposition \ref{prop:etarelax}(b).
\end{proof}

}}

We now make a remark about the special case of Proposition \ref{pr:DR.2} in which $\theta \in (0,\theta_0]$.
Indeed, in this case, $\tilde \theta = \theta$, and hence $\sigma=0$
 and $\tz_k=z_k$ for every $k \ge 1$. Thus,
the relaxed PR splitting method with $\theta \in (0,\theta_0]$
can be viewed as an exact non-Euclidean proximal point method with distance generating function $w$ as in \eqref{eq:DR-w-def}
with respect to the monotone inclusion $0 \in T(z) := (\cL_\theta+\gamma C)(z)$.
Note also that the latter inclusion depends on $\theta$.

As a consequence of Proposition \ref{pr:DR.2}, we are now ready to describe
 the pointwise and ergodic convergence rate for the relaxed PR splitting method.
 We first endow the space $\cZ := \cX \times \cX \times \cX$ with the semi-norm
$\|(u,v,x)\| : = \|x\|_\cX$ 
and hence Proposition \ref{propdualnorm} implies that
\beq \label{eq:norm-rel1}
\|(0,0,x)\|_* = 
\|x\|_\cX.
\eeq
It is also easy to see that the distance generating function $w$ defined in \eqref{eq:DR-w-def}
is in $\mathcal{D}_Z(m,M)$ with respect to $\|\cdot\|$
where $M=m=1/\theta$ (see Definition \ref{def:defw0}).  

Our next goal is to state a pointwise convergence rate bound for the relaxed PR splitting method. We start by stating a technical result
which is well-known for the case where $\beta=0$ (see for example Lemma 2.4 of \cite{He}). The proof for the general case,
i.e., $\beta \ge 0$, is similar and is given in the Appendix for the sake of completeness.

\begin{lemma}\label{lem:monotonicity}
Assume that $\theta\in (0,2 \theta_0]$. Then, for every $k \ge 2$, we have $\| \Delta x_{k}\|_{\cX} \le \| \Delta x_{k-1}\|_{\cX}$
where $\Delta x_k := x_k - x_{k-1}$.
\end{lemma}

We now state the pointwise convergence rate result for the relaxed PR splitting method.

\begin{theorem}\label{thm:pointwise2}
Consider the sequence $\{z_k = (u_k, v_k, x_k)\}$ generated by the relaxed PR splitting method with $\theta \in (0,2\theta_0)$.
Then, for every $k \ge 1$ and $z^*=(u^*,u^*,x^*) \in (\cL_{\tilde \theta}+\gamma C)^{-1}(0)$,
\[
a_k + b_k \in A(u_k) + B(v_k), \quad
\gamma \norm{ a_k +b_k}_\cX=\norm{ u_k - v_k}_\cX  \le
\frac{\sqrt{2}\|x_0-x^*\|_\cX}{\sqrt{k} \sqrt{2 \ttheta-\theta}}.
\]
\end{theorem}

\begin{proof}
{\color{black}{The inclusion and the equality in the theorem follows from \eqref{eq:DR-incl-aug'}.
Since by Proposition \ref{pr:DR.2}, the relaxed PR splitting method with $\theta \in (0,2\theta_0)$ is an NE-HPE instance
for solving the monotone inclusion $0 \in (\cL_{\tilde \theta}+\gamma C)(z)$ in which
$\sigma = (\theta/\tilde \theta-1)^2 <1$, $\varepsilon_k = 0$ and $\lambda_k=1$ for all $k \ge 1$,
it follows from Lemma \ref{lem:monotonicity},
Theorem \ref{th:alpha}, the fact that $M=m=1/\theta$, and relation \eqref{eq:dw0-def} 
that
\[
\|r_k\|_* \le \frac{\sqrt{2}M}{\sqrt{m}}(1+\sqrt{\sigma})
  \sqrt{\frac{(dw)_0}{1-\sigma}
\left(\frac{1}{ \sum_{j=1}^k \lambda_j^2}\right) } \le \frac{\sqrt{2}\|x_0-x^*\|_\cX}{\sqrt{k} \sqrt{2 \ttheta-\theta}}.
\]
The inequality of the theorem then follows {\color{black}{by 
Proposition \ref{pr:DR.2}(a) and relation  \eqref{eq:norm-rel1}.}}
}}
\end{proof}

Our main goal in the remaining part of this section is to derive ergodic convergence rate bounds for the relaxed PR splitting  method
for any $\theta \in (0, 2\theta_0]$. We start by stating the following variation of the transportation lemma for maximal
$\beta$-strongly monotone operators.


\begin{proposition}\label{prop:transportation}
Assume that $T$ is a maximal $\beta$-strongly monotone operator for some $\beta \ge 0$.
Assume also that $t_i \in T(u_i)$ for $i = 1, \ldots, k$, and  define
\beq
\bar{t}_k = \frac1k\sum_{i=1}^{k} t_i, \quad \bar{u}_k = \frac1k \sum_{i=1}^{k} u_i, \quad
\varepsilon_k = \frac1k \sum_{i=1}^k \langle  t_i - \beta u_i, u_i - \bar{u}_k \rangle_{\cX}
\eeq
Then, $\varepsilon_k \geq 0$ and $\bar{t}_k \in T^{[\varepsilon_k]}(\bar{u}_k)$.
\end{proposition}

\begin{proof}
The assumption that $T$ is a maximal $\beta$-strongly monotone operator implies that $T - \beta I$ is maximal monotone.  Hence, it follows
from the {{weak transportation formula}} (see Theorem 2.3 of \cite{Burachik1}) applied to $T - \beta I$ that $\varepsilon_k \geq 0$ and
$\bar{t}_k - \beta \bar{u}_k \in (T - \beta I)^{[\varepsilon_k]}(\bar{u}_k)$.
The result then follows by observing that $(T - \beta I)^{[\varepsilon_k]}(\bar{u}_k) +  \beta \bar{u}_k \subseteq T^{[\varepsilon_k]}(\bar{u}_k)$.
\end{proof}

In order to state the ergodic iteration complexity bound for the relaxed PR splitting  method, we introduce the ergodic sequences
\beq \label{eq:def-bak-bbk}
\bar u_k = \frac1k \sum_{i=1}^k u_i, \quad \bar v_k = \frac1k \sum_{i=1}^k v_i, \quad
\bar a_k = \frac1k \sum_{i=1}^k a_i, \quad \bar b_k = \frac1k \sum_{i=1}^k b_i
\eeq
and the scalar sequences
\beq \label{def:eps'''}
\varepsilon_k' := \frac1k \sum_{i=1}^k \inner{a_i - \beta u_i}{ u_i-\bar u_k}_{\cX}, \quad \varepsilon_k'' := \frac1k \sum_{i=1}^k \inner{b_i - \beta v_i}{ v_i-\bar v_k}_{\cX}.
\eeq

\begin{theorem}\label{thm:main}
 Assume that $\theta \in (0,2\theta_0]$ and
consider the ergodic sequences above. Then,
for every $k \ge 1$ and $z^*=(u^*,u^*,x^*) \in (\cL_{\tilde \theta}+\gamma C)^{-1}(0)$,
\[
\bar a_k \in A^{[\varepsilon_k']}(\bar u_k), \quad \bar b_k \in B^{[\varepsilon_k'']}(\bar v_k),
\]
\[
 \gamma \norm{ \bar a_k + \bar b_k}_\cX=\norm{ \bar u_k - \bar v_k}_\cX  \le \frac{2 \|x_0 - x^*\|_\cX}{k \theta}, \quad
\varepsilon_k' + \varepsilon_k'' \le  \frac{3(1+ 2(1-\tilde \theta/\theta)^2) \|x_0-x^*\|^2_\cX}{k \gamma \theta}.
\]
\end{theorem}

\begin{proof}
The first two inclusions follow from the two inclusions in (\ref{ak}) and (\ref{bk}), relation \eqref{eq:def-bak-bbk},
Assumption B0 and  Proposition \ref{prop:transportation}.
We will now derive the equality and the two inequalities of the theorem using the fact that the relaxed
PR splitting  method with $\theta \in (0,2\theta_0]$ is an instance of the NE-HPE method.
{\color{black}{Letting $\lambda_k=1$, $\varepsilon_k = 0$ for every $k$ and $\tilde{z}_k^a$, $r_k^a$ and $\varepsilon_k^a$ be as in
\eqref{SeqErg}, we easily see from {\color{black}{Proposition \ref{pr:DR.2}(a)}} and \eqref{eq:def-bak-bbk} that
\begin{align}
\label{identity1}
r^a_k &= (0,0,\bar u_k-\bar v_k) = \gamma (0,0,\bar a_k + \bar b_k), \\ 
\label{identity2}
\varepsilon_k^a & = \frac1k \sum_{i=1}^k \inner{r_i}{\tz_i-\tz_k^a}.
\end{align}
We claim that
\beq\label{identity4}
\varepsilon_k^a \geq  \gamma (\varepsilon_k' + \varepsilon_k'').
\eeq
Before proving this claim, we will use it to complete the proof of the theorem.
Indeed, using the definition of $w$ in \eqref{eq:DR-w-def}, relations \eqref{eq:dw0-def}, \eqref{eq:norm-rel1},
\eqref{identity1} and  \eqref{identity4}, the conclusion of Proposition \ref{pr:DR.2},
and Theorem \ref{a2947} with $T=\cL_{\tilde \theta} + \gamma C$, $M=m=1/\theta$
and $\lambda_k=1$ for all $k$, we conclude that
\[
\gamma \norm{ \bar a_k + \bar b_k}_\cX = \|\bar u_k-\bar v_k\|_\cX = \|r_k^a\|_* \le \frac {2 \sqrt{2}(dw)_0^{1/2}}{\sqrt{\theta} \Lambda_k}  \le \frac{2 \|x_0-x^*\|_\cX}{k \theta}
\]
and
\beq\label{identity5}
\gamma(\varepsilon_k' + \varepsilon_k'') \leq \varepsilon_k^a \le \frac{ 3(2 (dw)_0  + \rho_k)}{\Lambda_k}
\le 
 \frac{3 \left(  \|x_0-x^*\|^2_\cX + \rho_k \right)}{k \theta}
\eeq
where $\rho_k$ is defined in \eqref{eq:def-rhok}.
Moreover, using \eqref{eq:def-rhok}, the definition of $w$ in \eqref{eq:DR-w-def},
the definition of ${x}_i$ and $\tilde{x}_i$ in \eqref{eq:DR-iter-theta} and \eqref{eq:zk-theta}, respectively,
the triangle inequality, and Proposition \ref{proposition_desigualdades}(a), we conclude that
\begin{align*}
\rho_k :=& \displaystyle\max_{i=1,\ldots,k}(dw)_{z_{i}}(\tilde z_{i}) = \displaystyle\max_{i=1,\ldots,k} \frac{\|x_{i} - \tx_i\|_\cX^2}{2\theta} \\
 = & (1- \tilde \theta/\theta)^2 \displaystyle\max_{i=1,\ldots,k} \frac{\|x_{i} - x_{i-1}\|_\cX^2}{2\theta} 
\le  \frac{2(1-\tilde \theta/\theta)^2 \|x_0-x^*\|^2_\cX}{\theta}.
\end{align*}
The inequalities of the theorem now follows from the above three relations.

In the remaining part of the proof, we establish our previous claim \eqref{identity4}.
By {\color{black}{Proposition \ref{pr:DR.2}(a) and relations \eqref{eq:DR-incl-aug} and \eqref{identity2}}}, we have
\beq\label{identity3}
k \varepsilon_k^a = \sum_{i=1}^k \inner{r_i}{\tilde{z}_i - \tilde{z}_k^a} = \sum_{i=1}^k \inner{\cL_{\tilde \theta}(\tz_i)+ \gamma c_i}{\tz_i-\tz_k^a}
\eeq
where $c_i$ is defined in \eqref{eq:def-ck}.
Moreover, {\color{black}{we have 
\begin{align}
\sum_{i=1}^k \inner{\cL_{\tilde \theta}(\tz_i)+ \gamma c_i}{\tz_i-\tz_k^a} & =
\sum_{i=1}^k \inner{\cL_{\tilde \theta}(\tz_i - \tilde{z}_k^a)}{\tz_i-\tz_k^a}  + \gamma \sum_{i=1}^{k} \inner{c_i}{\tz_i-\tz_k^a} \nonumber \\
& = (1 - \tilde \theta) \sum_{i=1}^{k} \| (u_i - \bar{u}_k) - (v_i - \bar{v}_k) \|^2_{\cX} + \gamma \sum_{i=1}^k \inner{c_i}{\tz_i-\tz_k^a} \nonumber \\
&\ge  - \frac{\gamma\beta}2  \sum_{i=1}^{k} \| (u_i - \bar{u}_k) - (v_i - \bar{v}_k) \|^2_{\cX} +  \gamma \sum_{i=1}^k \inner{c_i}{\tz_i-\tz_k^a} , \label{ident1000}
\end{align}
where the second equality follows from \eqref{Ltildethetainequality} and the definitions of $\tilde{z}_k^a$ in \eqref{SeqErg}, $\tilde{z}_k$ in \eqref{eq:zk-theta}, and $\bar{u}_k$ and $\bar{v}_k$ in \eqref{eq:def-bak-bbk},
and the inequality follows from \eqref{eq:theta0} and the fact that $\tilde \theta \le \theta_0$ in view of  \eqref{eq:tildetheta}.
}}
Finally, using the definitions of $\tilde{z}_k^a$ in \eqref{SeqErg}, and $\tilde{z}_i$ and $c_i$ in Lemma \ref{lm:DR.1},
 and the straightforward relation
\begin{eqnarray*}
- \frac{1}{2} \sum_{i=1}^{k} \| (u_i - \bar{u}_k) - (v_i - \bar{v}_k) \|^2_{\cX} \geq - \sum_{i=1}^{k} \left( \inner{u_i}{u_i - \bar{u}_k}_{\cX} + \inner{v_i}{v_i - \bar{v}_k}_{\cX}\right),
\end{eqnarray*}
we conclude from \eqref{identity3} and \eqref{ident1000}  that
\[
\varepsilon_k^a \geq \frac{\gamma}{k} \sum_{i=1}^k ( \inner{a_i - \beta u_i}{u_i-\bar u_k}_{\cX} +  \inner{b_i - \beta v_i}{v_i-\bar v_k}_{\cX} ),
\]
and hence that the claim holds in view of \eqref{def:eps'''}.
}}
\end{proof}

We now make some remarks about the convergence rate bounds obtained in Theorem \ref{thm:main}.
In view of Lemma \ref{lm:sol}, $x^*$ depends on $\gamma$ according to
\[
x^* = \gamma a^*  + u^*, \quad a^* \in A(u^*) \cap -B(u^*).
\]
Hence, letting
\begin{align*}
d_0 &:= \inf \{ \| x_0 - u^* \|_\cX : u^* \in (A+B)^{-1}(0) \}, \\
S &:= \sup \{ \| a^*\|  :  a^* \in A(u^*) \cap -B(u^*), \, u^* \in (A+B)^{-1}(0) \},
\end{align*}
and assuming that $S< \infty$, it is easy to see that Theorem \ref{thm:main} and \eqref{eq:theta0} imply that
the relaxed PR splitting method with $\theta=2\theta_0$ satisfies
\[
\norm{ \bar a_k + \bar b_k}_\cX \le \frac{C_1(\gamma)}{\gamma k}, \quad \norm{ \bar u_k - \bar v_k}_\cX  \le \frac{C_1(\gamma)}{k}, \quad
\varepsilon_k' + \varepsilon_k'' \le  \frac{C_2(\gamma)}{k}
\]
where
\[
C_1(\gamma) =  C_1(\gamma;\beta,d_0) = {\Theta} \left(\frac{d_0 + \gamma S}{1+ \beta \gamma} \right), \quad
C_2(\gamma) = C_2(\gamma;\beta,d_0) = {\Theta} \left(\frac{(d_0 + \gamma S)^2}{\gamma(1+ \beta \gamma)} \right).
\]

When $S/\beta \ge  d_0$, then $\gamma = d_0/S$
minimizes both $C_1(\cdot)$ and $C_2(\cdot)$  up to a multiplicative constant, in which case
$C_1^* ={\Theta}(d_0)$, $C_1^*/\gamma= {\Theta}(S)$ and $C_2^*={\Theta}(Sd_0)$ where
\[
C^*_1 = C^*_1(\beta,d_0)  := \inf \{ C_1(\gamma) : \gamma>0\}, \quad
C^*_2 = C^*_2(\beta,d_0) := \inf \{ C_2(\gamma) : \gamma>0\}.
\]
 Note that this case includes the case in which $\beta=0$.
On the other hand, when $S/\beta < d_0$, then
both $C_1$ and $C_2$ are minimized up to a multiplicative constant by any $\gamma \ge d_0/S$, in which case
$C_1^* ={\Theta}(S/\beta)$ and $C_2^*={\Theta}(S^2/\beta)$. Clearly, in this case, $C_1^*/\gamma$ converges to zero as
$\gamma$ tends to infinity.


Indeed, assume first that $S/\beta \ge d_0$. Then, up to some multiplicative constants, we have
\[
C_1(\gamma) \ge \frac{d_0+\gamma S}{1+\beta \gamma} \ge \frac{d_0+\gamma S}{1+S \gamma/d_0} = d_0,
\]
\[
C_2(\gamma) \ge \frac{(d_0+\gamma S)^2}{\gamma(1+\beta \gamma)} \ge \frac{(d_0+\gamma S)^2}{\gamma(1+S \gamma/d_0)}
= \frac{d_0( d_0+\gamma S)}{\gamma}  = \frac{d_0^2}{\gamma}  + Sd_0,
\]
and hence that $C_1^*= \Omega(d_0)$ and $C_2^*={\Omega}(Sd_0)$.
Moreover, if $\gamma=d_0/S$, then the assumption $S/\beta \ge d_0$  implies that $\beta \gamma \le 1$, and hence that
$C_1^* = {\Theta}(d_0)$ and $C_2^*={\Theta}(Sd_0)$.

Assume now that $S/\beta < d_0$. Then, up to multiplicative constants,  it is easy to see that
\[
C_1(\gamma) \ge \frac{d_0+\gamma S}{1+\beta \gamma} \ge \frac{S}{\beta}
\]
\[
C_2(\gamma) \ge \frac{(d_0+\gamma S)^2}{\gamma(1+\beta \gamma)} \ge
\frac{(S/\beta+\gamma S)^2}{\gamma(1+\beta \gamma)} = \frac{S^2}{\gamma\beta^2} (1+\gamma\beta),
\]
and hence that $C_1^*= \Omega(S/\beta)$ and $C_2^*={\Omega}(S^2/\beta)$.
Moreover, if $\gamma \ge d_0/S$, then it is easy to see that $C_1^*={\Theta}(S/\beta)$ and
$C_2^* = {\Theta}(S^2/\beta)$.

Based on the above discussion, the choice $\gamma=d_0/S$ is optimal but has the disadvantage that
$d_0$ is generally difficult to compute. One possibility around this difficulty is  to use
$\gamma = D_0/S$ where $D_0$ is an upper bound on $d_0$.

{\color{black}{

\section{On the convergence of the relaxed PR splitting method }\label{discussions}

This section  discusses some new convergence results about the sequence generated by the relaxed PR splitting method
for the case in which $\beta>0$. It contains two subsections.
As observed in the Introduction, \cite{Dong} already establishes the convergence
of the relaxed PR sequence for the case in which $\beta \geq 0$ and $\theta<2\theta_0$.
The first subsection establishes convergence
of the relaxed PR sequence for the case in which $\beta>0$ and $\theta=2\theta_0$.
The second subsection describes an instance showing that the relaxed
PR spliting method may diverge when
$\beta \ge 0$ and $\theta \geq \min \{ 2(1+\gamma \beta), 2 + \gamma \beta + 1/(\gamma \beta) \}$. (Here,
we assume that $1/0=\infty$.) 
{\color{black}{Note that this instance, specialized to the case $\beta = 0$,
shows that the sequence $\{ z_k = (u_k, v_k, x_k) \}$ generated by the relaxed PR splitting method with $\beta=0$ may diverge
for any $\theta \ge 2$, and hence that the convergence result obtained  for any $\theta \in (0,2)$ in \cite{Dong} cannot be improved.
}}

\subsection{Convergence result about the relaxed PR sequence}\label{subsection1}

It is known that the sequence $\{z_k = (u_k, v_k, x_k)\}$ generated by the relaxed PR splitting method with
$\theta \in (0,2\theta_0)$ and $\beta \geq 0$ converges \cite{Dong}.
The main result of this subsection, namely Theorem \ref{prop:boundednessxk}, establishes convergence of this sequence
for $\theta = 2 \theta_0$ when $\beta > 0$. 

We start by giving a lemma which is used in the proof of Theorem \ref{prop:boundednessxk}.

\begin{lemma} \label{lm:boundedness}
Consider the sequence $\{z_k = (u_k, v_k, x_k)\}$ generated by the relaxed PR splitting method with $\theta \in (0,2\theta_0]$
and the sequence $\{\tilde z_k=(u_k,v_k,\tilde x_k)\}$ defined in \eqref{eq:zk-theta}. Then,
the sequences $\{ z_k \}$ and $\{ \tilde{z}_k \}$ are bounded.
\end{lemma}

\begin{proof}
The assumption that $\theta \in (0, 2\theta_0]$ together with the last conclusion of Proposition \ref{pr:DR.2} imply that
the relaxed PR splitting method is an NE-HPE instance with $\sigma \leq 1$.
 Hence, for any $z^* \in (\cL_{\tilde \theta}+\gamma C)^{-1}(0)$, it follows from
Proposition \ref{proposition_desigualdades}(a) that the sequence $\{(dw)_{z_k}(z^*)\}$ is non-increasing where
$w$ is the distance generating function given by \eqref{eq:DR-w-def}. Clearly, this observation implies that
$\{x_k\}$ is bounded. This conclusion together with  (\ref{eq:uv}) and the nonexpansiveness of
$J_{\gamma A}, J_{\gamma B}$ imply that $\{ u_k \}$ and $\{ v_k \}$ are also bounded.  Finally, $\{ \tilde{x}_k \}$ is bounded
due to the definition of $\tilde x_k$ in \eqref{eq:zk-theta}, and the boundedness of $\{x_k\}$, $\{ u_k \}$ and $\{ v_k \}$.
\end{proof}

As mentioned at the beginning of this subsection, the convergence of $\{(u_k,v_k)\}$ to some pair $(u^*,u^*)$
where $u^* \in (A+B)^{-1}(0)$ has been established in \cite{Dong} for the case in which
$\beta > 0$ and $\theta < 2 \theta_0$. The following result shows that the latter conclusion can also be
extended to $\theta = 2 \theta_0$.


\begin{theorem} \label{prop:boundednessxk}
In addition to Assumption B1, assume that Assumption B0 holds with $\beta>0$. Then, the sequence
$\{z_k = (u_k, v_k, x_k)\}$ generated by the relaxed PR splitting method with $\theta =2\theta_0$
converges  to some point lying in $ (\cL_{\theta_0}+\gamma C)^{-1}(0)$.
\end{theorem}

\begin{proof}
We assume that $\theta =2\theta_0$ and without any loss of generality that $\gamma=1$.
In view of \eqref{eq:tildetheta},  we have $\tilde{\theta} = \theta_0$.  
Let $z^\ast \in  (\cL_{\theta_0} + C)^{-1}(0)$.
Then, by Lemma \ref{lm:sol}, we have
$z^\ast = (u^\ast, u^\ast, x^\ast)$ where
\beq \label{eq:eqn1}
u^\ast = (A+B)^{-1}(0), \quad x^\ast - u^\ast \in A(u^\ast), \quad -x^\ast + u^\ast \in B(u^\ast).
\eeq
{\color{black}{Since $\tilde{\theta} = \theta_0$, it follows from {\color{black}{Proposition \ref{pr:DR.2}(b)}} that
$r_k \in (\cL_{\theta_0} + C)(\tilde{z}_k)$. 
This together with the fact that $0 \in (\cL_{\theta_0} + C)(z^\ast)$, inequality \eqref{Ltildethetainequality'} with
$(z,r)=(\tilde z_k,r_k)$ and $(z',r')=(z^*,0)$, and the fact that $\theta_0 = 1 + \beta/2$, then imply that
\begin{align}
\langle r_k, \tilde{z}_k - z^\ast \rangle & 
\geq (1 - \theta_0) \| (u_k - u^\ast) - (v_k - u^\ast)  \|_{\cX}^2 + \beta (\|u_k - u^\ast \|^2_{\cX} + \| v_k - u^\ast \|^2_{\cX} ) \nonumber \\ 
& = \frac{\beta}{2} \| u_k + v_k - 2u^\ast \|^2_{\cX} \geq 0. \label{eqn2}
\end{align}
}}
{\color{black}{Since the last conclusion of Proposition \ref{pr:DR.2} states that the relaxed PR splitting method with $\theta=2\theta_0$ is
an NE-HPE instance with respect to the monotone inclusion $0 \in (\cL_{\tilde \theta}+\gamma C)(z)$ in which
$\sigma=1$, $\lambda_k = 1$ and $ \varepsilon_k = 0$ for every $k$, it follows from Proposition \ref{proposition_desigualdades}(b), \eqref{eqn2} and the assumption that $\beta>0$ that
\beq
\lim_{k \rightarrow \infty} \langle r_{k}, \tilde{z}_{k} - z^\ast \rangle =
\lim_{k \rightarrow \infty} \| u_{k} + v_{k} - 2u^\ast \|_\cX = 0. \label{limit1}
\eeq
By Lemma \ref{lm:boundedness}, $\{z_k\}$, and hence $\{x_k\}$, is bounded.  Therefore, }}there exist an infinite index set ${\cal K}$ and $\bar x \in {\cX}$ such that
$\lim_{k \in {\cal K}} x_{k-1} = \bar x$,
from which we conclude that
\beq \label{eq:defuvbar}
\lim_{k \in {\cal K}} u_k = \bar u := J_A(\bar x) , \quad
\lim_{k \in {\cal K}} v_k = \bar v := J_B( 2 \bar u - \bar x)
\eeq
in view of  \eqref{eq:uv}, \eqref{eq:DR-iter-theta} and the continuity of the point-to-point maps $J_A$ and $J_B$.
Clearly, relations \eqref{eq:DR-iter-theta}, \eqref{limit1} and \eqref{eq:defuvbar},  {\color{black}{Proposition \ref{pr:DR.2}(a)}}, the definitions of
$J_B$ following \eqref{eq:seqqq} and $\tz_k$ in \eqref{eq:zk-theta}, and the fact that $\tilde{\theta} = \theta_0$,
imply that
\begin{align}
 \label{eq:notob}
& \bar u + \bar v = 2 u^*, \quad 2 \bar u - \bar v - \bar x \in B(\bar v), \quad
\lim_{k \in \cK}  z_k = (\bar u, \bar v, \bar x+ \theta (\bar v-\bar u)), \\
& \lim_{k \in \cK} r_k= (0,0,\bar u-\bar v), \quad
\lim_{k \in \cK}  \tilde z_k = \tz := (\bar u, \bar v, \bar x+ \theta_0 (\bar v-\bar u)). \label{notob1}
\end{align}
Clearly, \eqref{limit1} and \eqref{notob1} imply that
\beq \label{notob2}
0 =  \lim_{k \in \cK} \langle r_{k}, \tilde{z}_{k} - z^\ast \rangle =
\langle \bar u - \bar v,  \bar x+ \theta_0 (\bar v-\bar u) - x^\ast \rangle_{\cX} =
- \theta_0 \| \bar u- \bar v\|^2_{\cX} + \langle \bar u - \bar v , \bar x - x^* \rangle_{\cX}.
\eeq
Using the second inclusion in \eqref{eq:eqn1}, the identity and the inclusion in \eqref{eq:notob},
the $\beta$-strong monotonicity of $B$, and relation \eqref{notob2},
we then conclude that
\begin{align*}
\frac{\beta}4 \| \bar v - \bar u\|^2_{\cX} &= \beta \| \bar v - u^\ast \|^2_{\cX} \le
\left\langle (2 \bar u - \bar v - \bar x) - ( u^\ast - x^\ast), \bar v - u^\ast  \right \rangle_{\cX}
= \frac12 \left \langle  \frac32 (\bar u - \bar v) - \bar x + x^\ast,  \bar v - \bar u \right \rangle_{\cX} \\
&= \frac12 \left \langle  \bar x - x^\ast,  \bar u - \bar v \right \rangle_{\cX} - \frac34 \| \bar v - \bar u\|^2_{\cX}
= \left(  \frac{\theta_0}2 - \frac34 \right) \| \bar v - \bar u\|^2_{\cX}.
\end{align*}
The latter inequality together with the fact that $\theta_0=1+(\beta/2)$ then
imply that $\bar u=\bar v = u^*$ where the last equality is due to
the identity in \eqref{eq:notob}.
We have thus shown that $\{u_k\}_{k \in \cK}$ and $\{v_k\}_{k \in \cK}$ both converge to  $u^* = (A+B)^{-1}(0)$.  Since $\bar u=\bar v=u^*$, it follows from \eqref{eq:notob} and
\eqref{notob1} that
\begin{align*}
\lim_{k \in \cK}r_k=0, \quad \lim_{k \in \cK} z_k = \lim_{k \in \cK} \tz_k = \tz = (u^*,u^*,\bar x), \\
\lim_{k \in \cK} (dw)_{z_k}(\tz_k) = \lim_{k \in \cK} \|x_k-\tx_k\|_\cX^2/(2\theta) = \|\bar x-\bar x\|_\cX^2/(2\theta)=0.
\end{align*}
Hence,  Proposition \ref{lm:convergence} with
$T=(\cL_{\theta_0} + \gamma C)$ implies that $\tz \in ({\cal L}_{\theta_0} + \gamma C)^{-1}(0)$ and
$0=\lim_{k \to \infty} (dw)_{z_k}(\tz) = \lim_{k \to \infty} \|\bar x-x_k\|^2_\cX/(2\theta)$.
We thus conclude that $\{z_k\}$ converges to $(u^*,u^*, \bar x)=\tz \in ({\cal L}_{\theta_0} + \gamma C)^{-1}(0)$.
\end{proof}


Before ending this subsection, we make two remarks.
First, for a fixed $\tau >0$, consider the set 
\[
R(\tau):= \{ (\beta,\theta) \in \R^2 : \beta>0, \, 0 < \theta \le 2 + \tau \beta \}.
\]
Then, it follows from Theorem \ref{prop:boundednessxk} and the observation in the paragraph preceding it
that the sequence generated by the relaxed PR splitting method {\color{black}{with relaxation parameter $\theta$ to solve \eqref{eq:mainincl} with $A, B$ maximal $\beta$-strongly monotone}} converges for any $(\beta,\theta) \in R(1)$.
Second, it follows from the example presented in the next subsection that
the above conclusion fails if $R(1)$ is enlarged to the region $R(\tau)$ for any  $\tau>1$.

}}

{\color{black}{

\subsection{Non-convergent instances for $\theta \ge \min \{ 2+2\gamma \beta, 2 + \gamma \beta + 1/(\gamma \beta) \}$ }\label{subsection2}

By  \cite{Dong} and Theorem \ref{prop:boundednessxk}, the sequence $\{ x_k \}$
generated by the relaxed PR splitting  method converges whenever either
$\theta \in (0, 2 + \gamma \beta)$ or $\theta = 2 + \gamma \beta$ and $\beta > 0$.
This subsection gives an instance of \eqref{eq:mainincl}, {\color{black}{where $A, B$ are maximal $\beta$-strongly monotone}}, for which the sequence $\{x_k\}$ generated by the  relaxed PR splitting  method {\color{black}{with relaxation parameter $\theta$}} does not converge
when $\beta \ge 0$ and $\theta \geq \min \{ 2(1+\gamma \beta), 2 + \gamma \beta + 1/(\gamma \beta) \}$.  

Recall from (\ref{eq:uv}) and (\ref{eq:DR-iter-theta}) that the relaxed PR splitting  method iterates as
\begin{eqnarray}\label{PR2}
x_{k+1} = x_{k} + \theta (J_{\gamma B}(2J_{\gamma A}(x_{k}) - x_{k}) - J_{\gamma A}(x_{k}))
\end{eqnarray}
where $\theta > 0$.  Without any loss of generality, we assume that $\gamma = 1$ in (\ref{PR2}).

We now describe our instance. First, let $\cX := \mathcal{\tilde{X}} \times \mathcal{\tilde{X}}$ where $\mathcal{\tilde{X}}$ is a finite-dimensional real vector space,
and let $A_0, B_0: \cX \tos \cX$ be defined as
\[
A_0(\tx_1,\tx_2) = (0,0), \quad B_0(\tx_1,\tx_2)= N_{\{0\}}(\tx_1) \times \{0\}, \quad
\forall x=(\tx_1,\tx_2) \in \mathcal{\tilde{X}} \times \mathcal{\tilde{X}}
\]
where $N_{\{0\}}(\cdot)$ denotes the normal cone operator of the set $\{0\}$.
Clearly, $A_0$ and $B_0$ are both maximal monotone operators and
\[
J_{A_0}(\tilde{x}_1, \tilde{x}_2) = (\tilde{x}_1,\tilde{x}_2), \quad
J_{B_0}(\tilde{x}_1,\tilde{x}_2) = (0,\tilde{x}_2), \quad \forall x=(\tilde{x}_1, \tilde{x}_2) \in \mathcal{\tilde{X}} \times \mathcal{\tilde{X}}.
\]
Now define
  $A := A_0 + \beta I$ and $B:=B_0+\bar{\beta} I$ where $\bar{\beta} \geq \beta \ge 0$.
It follows that $A$ is a $\beta$-strongly maximal monotone operator and $B$ is a $\bar{\beta}$-strongly maximal monotone operator.
Hence, the instance we are describing is slightly more general in that
$A$ and $B$ have different  strong monotonicity parameters.
Moreover, for any $x =({\tx}_1,{\tx}_2) \in \cX$, it is easy to see that
\begin{eqnarray*}
& & J_A(x) = J_{\frac{1}{1+\beta}A_0}\left(\frac{1}{1 + \beta}({\tx}_1,{\tx}_2)\right) = \frac{1}{1 + \beta}({\tx}_1,{\tx}_2),\\
& & J_B(x) = J_{\frac{1}{1+\beta}B_0}\left(\frac{1}{1 + \bar{\beta}}({\tx}_1,{\tx}_2)\right) = \frac{1}{1+\bar{\beta}} (0, {\tx}_2).
\end{eqnarray*}
and hence that
\begin{eqnarray}\label{relaxedPRequality}
x + \theta [J_{B}(2J_{A}(x) - x) - J_{A}(x)] = \left( \frac{(1 + \beta - \theta)}{1 + \beta}\tilde{x}_1,  \left[1 - \frac{(\beta + \bar{\beta}) \theta}{(1+\beta)(1+ \bar{\beta})}\right] \tilde{x}_2 \right).
\end{eqnarray}

From \eqref{PR2} and (\ref{relaxedPRequality}), we easily see that  the sequence $\{x_k\}$ generated by the relaxed 
PR splitting method diverges whenever
\begin{eqnarray*}
\frac{(1 + \beta - \theta)}{1 + \beta} \leq -1, \ \ \mbox{\rm or}
\ \ 1 - \frac{(\beta + \bar{\beta}) \theta}{(1+\beta)(1+ \bar{\beta})} \le -1,
\end{eqnarray*}
or equivalently, whenever
\[
\theta \geq \min \left\{2(1+\beta) \,,\, 2 + \frac{2(1+\beta\bar{\beta})}{\beta + \bar{\beta}} \right\}.
\]
Note that when $\beta=\bar \beta$, the above inequality reduces to $\theta \geq \min \{2(1+\beta), 2 + \beta + 1/\beta \}$.

Before ending this subsection, we make two  remarks.
First, when $\beta=0$ and hence $A$ is not strongly monotone, the sequence $\{ x_k \}$ for the above example diverges for any $\theta \ge 2$ even if
$B$ is strongly monotone, i.e., $\bar \beta>0$.
Second, the above example specialized to the case in which $\beta=\bar \beta$ easily shows that the sequence generated by
the relaxed PR splitting method does not necessarily converge for any $(\beta,\theta) \in R(\tau)$ if $\tau>1$
where $R(\tau)$ is defined at the end of Subsection \ref{subsection1}.

}}

{\color{black}{
\section{Numerical study}\label{numerical}

This section illustrates the behavior of the relaxed PR splitting method for
solving the weighted Lasso minimization problem \cite{Giselsson} (see also \cite{Candes})
\begin{eqnarray}\label{Lasso}
\min_{u \in \R^n} f(u) + g(u),
\end{eqnarray}
where $f(u) := \frac{1}{2} \| Cu - b \|_{\cX}^2$ and $g(u) := \| Wu \|_1$  for every $u \in \R^n$.
Our numerical experiments consider instances where
$n = 200$, $b \in \R^{300}$ and $C \in \R^{300 \times 200}$ is a sparse data matrix with an average of $10$ nonzero entries per row.  Each component of $b$ and each nonzero element of $C$ is drawn from a Gaussian distribution with zero mean and unit variance,
while $W \in \R^{200 \times 200}$ is a diagonal matrix with positive diagonal elements drawn from a uniform distribution on the interval
$[0,1]$.
This setup follows that of \cite{Giselsson}.  Note that $\cX = \R^{300}$ and $\| \cdot \|_1$ is the 1-norm on $\R^{200}$.
Observe that $f$ is $\alpha$-strongly convex\footnote{A function $f$ is $\alpha$-strongly convex on
$\cX$ if $f - \frac{\alpha}{2} \| \cdot \|_{\cX}^2$ is convex on $\cX$.} on $\R^{200}$ where $\alpha = \lambda_{\min}(C^TC)$
is the minimum eigenvalue of $C^T C$.
Also, $f$ is differentiable and its gradient is $\kappa$-Lipschitz continuous 
on $\R^{200}$ where $\kappa = \lambda_{\max}(C^TC)$ is the maximum eigenvalue of $C^TC$.  The function $g$ is clearly convex on $\R^{200}$.


We consider solving \eqref{Lasso} by apply the relaxed PR splitting method (\ref{PR2}) to solve
the inclusion \eqref{eq:mainincl} with
\begin{eqnarray}
A = \partial f - \alpha^\prime I, \quad B = \partial g + \alpha^\prime I, \label{eq:parti}
\end{eqnarray}
where $0 \leq \alpha^\prime \leq \alpha = \lambda_{\min}(C^TC)$.
Since $A$ (resp., $B$) is $(\alpha-\alpha')$-strongly (resp., $\alpha'$-strongly) maximal monotone,
the results developed in Sections \ref{sec:PRmethod} and \ref{discussions} for the relaxed PR splitting method with $(A,B)$ as above applies with
$\beta  = \min \{ \alpha - \alpha^\prime, \alpha^\prime \}$.
Our goal in this section is to gain some intuition of how the relaxed PR splitting method performs as
$\alpha^\prime$ (and hence $\beta$), $\gamma$ and $\theta$ change.
In our numerical experiments, we start  the relaxed PR splitting algorithm with $x_0 = 0$ and terminate it
when $\| x_{k+1} - x_{k} \|_{\cX} \leq 10^{-5}$.
The paragraphs below report the results of  three experiments.

In the first experiment, we generate 100 random instances of $(C,W,b)$ and we observed that the condition $\lambda_{\min}(C^T C) > 0$
holds for all instances.  The relaxed
PR splitting method is used to solve these instances of \eqref{Lasso} for various values of $\theta$ and with the pair
$(\gamma,\alpha')$ taking on the values $(1,0)$, $(1,\alpha/2)$, $(1/\sqrt{\alpha\kappa},0)$ and $(1/\sqrt{\alpha\kappa},\alpha/2)$.
Note that it follows from Proposition $3$ of \cite{Giselsson} that when $\alpha'=0$ and $\theta=2$,
the choice of $\gamma =1/\sqrt{\alpha\kappa}$ has been shown to be optimal for the
relaxed PR splitting method.
 Our results are shown in Table \ref{table3}.
\begin{table}[htpb]
\centering
\begin{tabular}{||c|c|c||c|c||}  \hline
        & \multicolumn{4}{|c||}{Average Number of Iterations} \\ \cline{2-5}
 $\theta$         & \multicolumn{2}{|c||}{$\gamma = 1$} & \multicolumn{2}{|c||}{$\gamma = 1/\sqrt{\alpha \kappa}$} \\ \cline{2-5}
          & $\alpha^\prime = 0$ & $\alpha^\prime = \alpha/2$   & $\alpha^\prime = 0$ & $\alpha^\prime = \alpha/2$  \\ \hline \hline
  $1$       &   $141.79$              &  $140.64$      & $60.10$   &  $60.11$                \\ \hline
  $1.25$    &   $115.96$              &  $115.06$      & $48.47$   &  $48.48$               \\ \hline
  $1.50$    &   $98.31$               &  $97.48$       & $40.51$   &  $40.49$               \\ \hline
  $1.75$    &   $85.33$               &  $84.64$       & $34.67$   &  $34.70$               \\ \hline
  $2$       &   $264.80$              &  $75.08$       & $58.54$   &  $42.11$             \\ \hline
  $2 + \gamma \alpha/2$    &  $> 500$    &  $73.25$     & $74.73$   &  $49.60$            \\ \hline
\end{tabular}
\caption{Average number of iterations performed by the relaxed PR splitting method \eqref{PR2} to solve \eqref{Lasso} based on
the partition $(A,B)$ given by \eqref{eq:parti} for 4 pairs $(\gamma,\alpha')$ and 6 different values of $\theta$.}\label{table3}
\end{table}
We see from the table that, except when $\theta = 2$ and $\theta = 2 + \gamma \alpha/2$, the average number of iterations for
$\alpha^\prime = 0$ and $\alpha^\prime = \alpha/2$ are similar.
However, when $\theta = 2$ and $\theta = 2 + \gamma \alpha/2$, the choice $\alpha^\prime = \alpha/2$ outperforms the
one with $\alpha^\prime = 0$.
One possible explanation for this behavior is due to the fact that when $\theta = 2$ and $\theta = 2 + \gamma \alpha/2$, the relaxed PR sequence
converges when both operators are strongly monotone, while it does not necessary converge when either one of the operators is only monotone.  Note also that
the results in the last row of the table confirm the convergence result of the relaxed PR splitting method
(see Theorem \ref{prop:boundednessxk})
for the case in which $A$ and $B$ are $\beta$-strongly maximal monotone operators with $\beta > 0$ and
$\theta = 2 + \gamma \beta$.  Finally, our results (the last two rows of table) suggest that, if $A$ is maximal $\alpha$-strongly monotone and $B$ is only maximal monotone, it might be advantageous to use the relaxed PR splitting method  with
$0 < \alpha^\prime < \alpha$ (and hence $\beta>0$) instead of $\alpha'= 0$ (and hence $\beta=0$).

In our second experiment, we use relaxed PR splitting method with $(\theta,\gamma)$ equal to $(2,1)$ and $(2,1/\sqrt{\alpha\kappa})$,
and with $\alpha^\prime$ varying from $0$ to $\alpha$,
to solve \eqref{Lasso} for a randomly generated $(C,W,b)$.
In this instance, $\alpha = \lambda_{\min}(C^T C) = 0.3792$ and $\kappa = \lambda_{\max}(C^T C) = 57.6624$.
Our results are shown in Figure \ref{figure1}.
\begin{figure}[htpb]
\centering
\includegraphics[width=8cm]{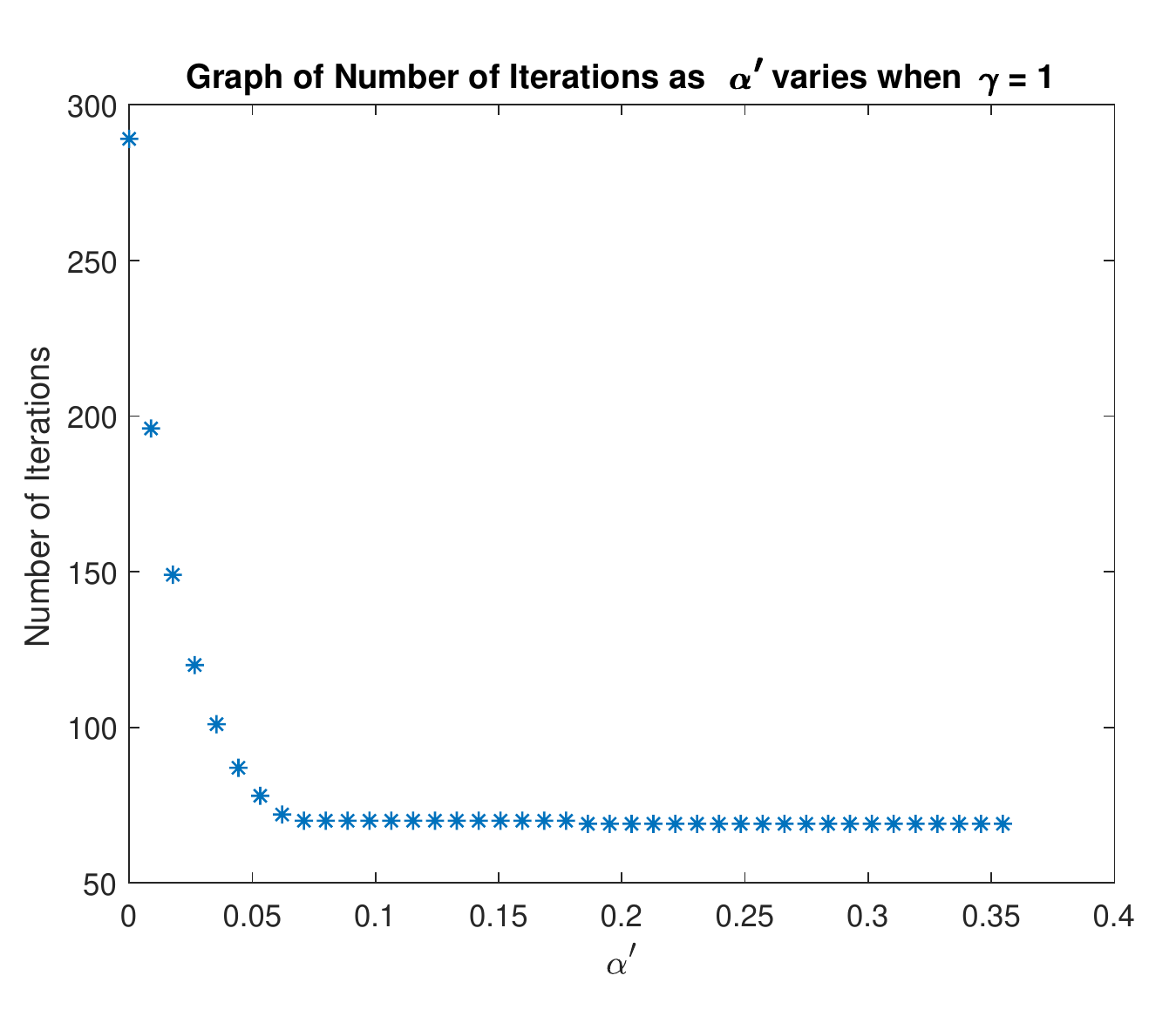}
\quad
\includegraphics[width=8cm]{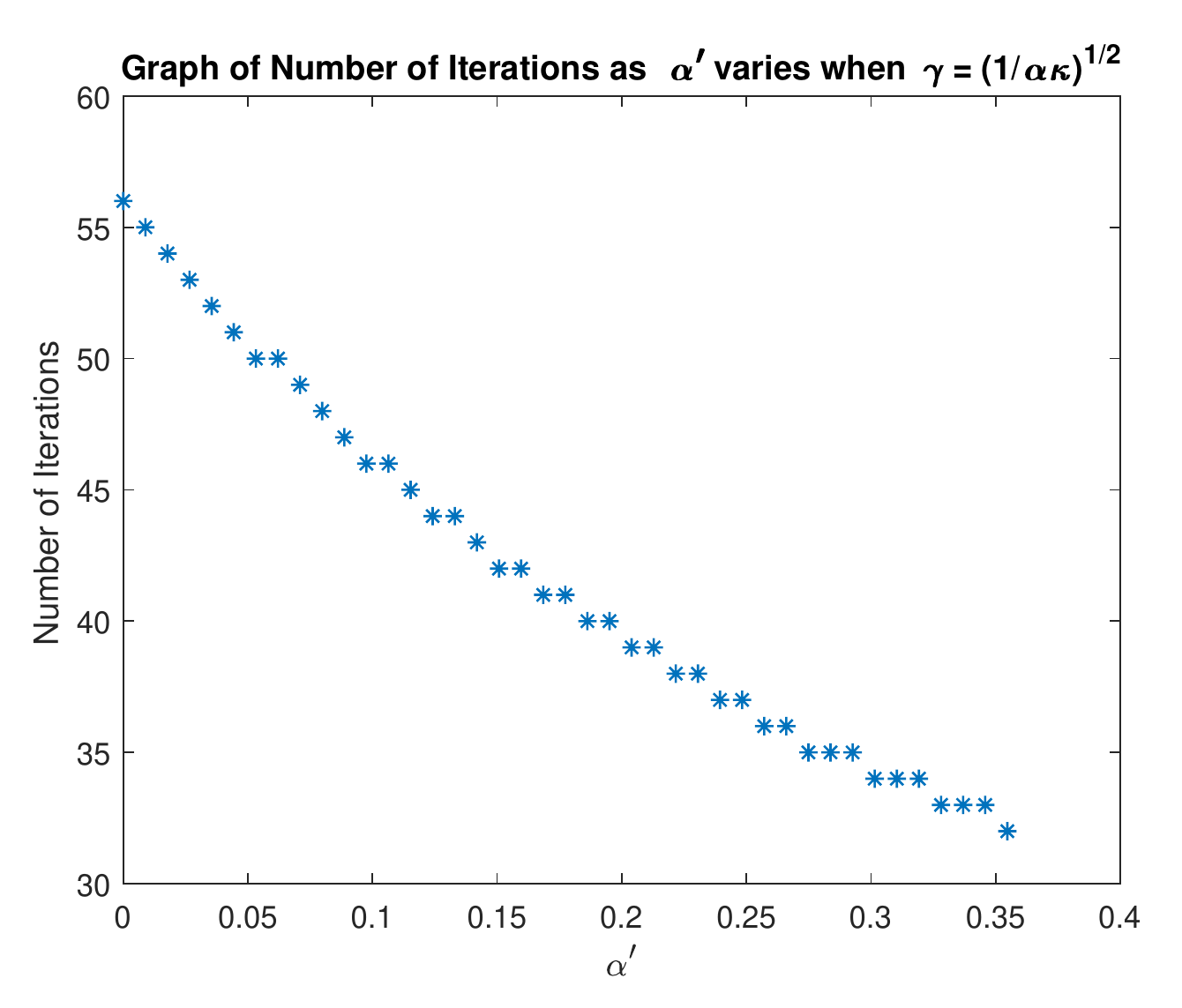}
\centering
\caption{Two graphs showing how the number of iterations performed by (\ref{PR2}) changes with varying $\alpha^\prime$ using
the partition  $(A,B)$ given \eqref{eq:parti}.} \label{figure1}
\end{figure}
We see from Figure \ref{figure1} that the number of iterations decreases as $\alpha^\prime$ increases in both cases.
 These graphs again suggest that it might be advantageous to have
 $A$ and $B$ maximal $\beta$-strongly monotone with $\beta > 0$.
We also observe that as $\alpha^\prime$ approaches $\alpha$, the number of iterations does not increase even though
 the operator $A$ is losing its strong monotonicity.  


In our third experiment, we performed the same numerical experiments as the ones mentioned above but
with $(A,B) = (\partial g + \alpha^\prime I, \partial f - \alpha^\prime I)$ instead of
$(A,B) = (\partial f - \alpha^\prime I,  \partial g + \alpha^\prime I)$ and note that the results obtained were very similar
to the ones reported above.
Hence interchanging $A$ and $B$ in the implementation of the relaxed PR splitting method have little
impact on its performance.

}}

\section{Concluding remarks}\label{conclusion}

{\color{black}{
This paper establishes convergence of the sequence of iterates
and an ${\cal O}(1/{k})$ ergodic convergence rate bound  for the relaxed PR splitting method
for any $\theta \in (0, 2 + \gamma \beta]$ by
viewing it as an instance of a non-Euclidean HPE framework. It also establishes
an ${\cal O}(1/\sqrt{k})$ pointwise convergence rate bound for it for any $\theta \in (0, 2 + \gamma \beta)$.
Furthermore, an example showing that PR iterates 
do not necessarily converge for $\beta \ge 0$ and $\theta \geq \min \{ 2(1+\gamma \beta), 2 + \gamma \beta + 1/(\gamma \beta) \}$
is given.

Table \ref{table1} (resp., Table \ref{table2}) for the case in which  $\beta=0$ (resp., $\beta>0$) provides a summary of the convergence rate results known so far for the relaxed PR splitting method when
$(A,B)=( \partial f, \partial g)$ for some convex functions $f$ and $g$.  However, we observe that some of these
results also hold for pairs $(A,B)$ of maximal monotone operators which are not subdifferentials.
The term ``R-linear" in the tables below stands for linear convergence of the sequence $\{ x_k \}$ generated by the relaxed PR splitting method.  

\begin{table}[htpb]
\centering
\begin{tabular}{||c|c|c|c|c||}  \hline
\multicolumn{4}{||c|}{Convergence}   & Additional conditions on $f, g$  \\\cline{1-4}
 $\theta$  &  Rate &  Type      & Reference   & besides convexity \\ \hline \hline
 $(0,2)$  & $\mathcal{O}(1/\sqrt{k})$ & Pointwise  &  \cite{He}  &  $-$  \\ \hline
 $(0,2)$ & $o(1/\sqrt{k})$ & Pointwise &   \cite{Davis} & $-$ \\ \hline \hline
 $(0,2]$  & $\mathcal{O}(1/k)$ &  Ergodic        &  \cite{Davis}            &  $-$  \\ \hline \hline
   $(0,2)$ & $o(1/k)$       & Best iterate   & \cite{Davis2} & $\nabla f$ or $\nabla g$ Lipschitz \\ \hline \hline
 $1,2$      & $-$    & R-linear   & \cite{MR551319}  & $f$ strongly convex and $\nabla f$ Lipschitz \\ \hline
 $(0,2)$    & $-$  & R-linear & \cite{Giselsson2} & $g$ strongly convex and $\nabla f$ Lipschitz 
\\ \hline
 $(0,2]$ & $-$ & R-linear & \cite{Davis2} & ($f$ or $g$ strongly convex) and  \\
        &     &           &                &    ($\nabla f$ or $\nabla g$ Lipschitz)             \\ \hline
 $(0,2]$ & $-$ & R-linear & \cite{Dong} & ($f$ or $g$ strongly convex) and  \\
        &     &           &                &    $\nabla f$ Lipschitz            \\ \hline        
 $(0,\frac{4}{1+\delta})$  &  $-$ & R-linear & \cite{Giselsson} (resp., \cite{Giselsson2})  & $f$ (resp., $g$) strongly convex and
 \\
 $\delta \in (0,1)$        &                        &   &                  & $\nabla f$ (resp., $\nabla g$)
Lipschitz  \\ \hline
\end{tabular}
\caption{Summary of convergence rate results for the relaxed PR splitting method on the sum of two convex functions $f, g$.}\label{table1}
\end{table}



\begin{table}[htpb]
\centering
\begin{tabular}{||c|c|c|c|c||}  \hline
\multicolumn{4}{||c|}{Convergence}   &  Additional conditions on $f, g$  \\\cline{1-4}
 $\theta$ & Rate &  Type &   Reference        & besides strong convexity \\ \hline \hline
 $(0,2 + \gamma \beta)$  & $\mathcal{O}(1/\sqrt{k})$ & Pointwise & this paper &  $-$  \\ \hline \hline
 $(0,2+ \gamma \beta]$  & $\mathcal{O}(1/k)$ & Ergodic & this paper  &  $-$ \\ \hline \hline
 $(0,2]$                & $-$ & R-linear & \cite{Dong} & $\nabla f$ Lipschitz \\  \hline
\end{tabular}
\caption{Summary of convergence rate results for the relaxed PR splitting method on the sum of two $\beta$-strongly convex functions $f, g$, with $\beta > 0$.}\label{table2}
\end{table}
}}
We observe that our analysis in Sections \ref{sec:PRmethod} and \ref{discussions}, in contrast to the ones in \cite{Davis2,Giselsson2,Giselsson,MR551319}, does not impose any regularity condition on $A$ and $B$ such as assuming one
of them to be a Lipschitz, and hence point-to-point, operator. Also, if only one of the operators, say $A$, is assumed to be maximal
$\beta$-strongly monotone,
\eqref{eq:mainincl} is equivalent to $0\in(A'+B')(u)$ where $A':=A-(\beta/2)I$  and $B':=B+(\beta/2)I$ are now both $(\beta/2)$-strongly monotone.
Thus, to solve \eqref{eq:mainincl},  the relaxed PR method with $(A,B)$ replaced by $(A',B')$ can be applied,
thereby ensuring convergence of the iterates, as well as pointwise and ergodic
convergence rate bounds, for values of $\theta \ge 2$.
{\color{black}{This idea was tested in our computational results of Section \ref{numerical} where
a weighted Lasso minimization problem \cite{Candes} is solved using the partitions $(A,B)$ and $(A',B')$.
Our conclusion is that the partition $(A',B')$  substantially outperforms  the other one for values of $\theta \ge 2$.

}}



{\footnotesize

\bibliographystyle{plain}
\bibliography{RADMM_ref2}
}

\section*{Appendix}

\noindent {\bf{Proof of Lemma \ref{lem:monotonicity}}}:  To simplify notation, let
\[
\Delta x=\Delta x_k, \quad \Delta x^- := \Delta x_{k-1}, \quad \Delta u = u_k-u_{k-1}, \quad \Delta v = v_k-v_{k-1}, \quad
\Delta a = a_k-a_{k-1}, \quad \Delta b = b_k-b_{k-1}.
\]
Then, it follows from the second identity in \eqref{eq:DR-iter-theta} and relation \eqref{eq:def-ak-bk} that
\begin{eqnarray}\label{equalities}
\Delta x = \Delta x^- + \theta (\Delta v - \Delta u), \quad \gamma \Delta a = \Delta x^- - \Delta u, \quad
\gamma \Delta b = 2 \Delta u - \Delta v - \Delta x^-.
\end{eqnarray}
Also, the two inclusions in \eqref{ak} and \eqref{bk} together with the $\beta$-strong monotonicity of $A$ and $B$ imply that
\[
\inner{\Delta a}{\Delta u}_{\cX} \ge \beta \|\Delta u\|^2_{\cX}, \quad \inner{\Delta b}{\Delta v}_{\cX} \ge \beta \|\Delta v\|^2_{\cX}.
\]
Combining the last two identities in \eqref{equalities} with the above inequalities, we obtain
\[
\inner{ \Delta x^- - \Delta u}{\Delta u}_{\cX} \ge \gamma \beta \| \Delta u\|^2_{\cX}, \quad
\inner{ 2 \Delta u - \Delta v - \Delta x^-}{\Delta v}_{\cX} \ge \gamma \beta \| \Delta v\|^2_{\cX}.
\]
Adding these two last inequalities and simplifying the resulting expression, we obtain
\begin{eqnarray}\label{importantinequality}
\inner{ \Delta x^-}{\Delta u - \Delta v}_{\cX} + 2 \inner{\Delta u}{\Delta v}_{\cX} \ge (1+ \gamma \beta) [ \| \Delta u\|^2_{\cX} + \| \Delta v\|^2_{\cX} ]
\end{eqnarray}
From the first equality in (\ref{equalities}), we have
\begin{eqnarray*}
2\theta \inner{ \Delta x^-}{\Delta u - \Delta v}_{\cX} = \| \Delta x^- \|^2_{\cX} - \| \Delta x \|^2_{\cX} + \theta^2 \| \Delta v - \Delta u \|^2_{\cX},
\end{eqnarray*}
which upon substituting into (\ref{importantinequality}), the following is true:
\begin{eqnarray*}
\| \Delta x^- \|^2_{\cX} - \| \Delta x \|^2_{\cX} \geq 2 \theta \left( \left[1 + \gamma \beta - \frac{\theta}{2}\right]( \| \Delta u \|^2_{\cX} + \| \Delta v \|^2_{\cX} ) + 2 \left[ \frac{\theta}{2}  - 1 \right] \inner{\Delta u}{\Delta v}_{\cX} \right).
\end{eqnarray*}
Note that the right-hand side in the above inequality is greater than or equal to zero if $\theta \in (0, 2\theta_0]$.  Hence, we have if $\theta \in (0, 2\theta_0]$,
\begin{eqnarray*}
\| \Delta x \|_{\cX} \leq \| \Delta x^- \|_{\cX}.
\end{eqnarray*}
{{\ \hfill\hbox{%
      \vrule width1.0ex height1.0ex
    }\parfillskip 0pt}\par}

\end{document}